\documentclass{amsart}
\usepackage{amssymb}
\usepackage{amsfonts}
\usepackage{amsmath}
\usepackage[usenames,dvipsnames]{color}
\usepackage{tikz}
\usetikzlibrary{matrix,arrows,decorations.pathmorphing}
\usepackage{amsfonts}
\usepackage{amsthm}
\usepackage{enumerate}
\usepackage{mathrsfs}
\usepackage{bbm} 
\usepackage{draftwatermark}
\SetWatermarkScale{4}
\SetWatermarkLightness{1}
\SetWatermarkText{Preprint}
\DeclareMathAlphabet{\mathpzc}{OT1}{pzc}{m}{it}
\newtheorem{theorem}{Theorem}[section]

\newtheorem{corollary}[theorem]{Corollary}

\newtheorem{definition}[theorem]{Definition}

\newtheorem{lemma}[theorem]{Lemma}

\newtheorem{proposition}[theorem]{Proposition}
\newtheorem{remark}[theorem]{Remark}

\newtheorem*{main}{Main Theorem}

\providecommand{\intersect}{\ensuremath{\cap}}
\providecommand{\Intersect}{\ensuremath{\bigcap}}
\providecommand{\union}{\ensuremath{\cup}}
\providecommand{\Union}{\ensuremath{\displaystyle\bigcup}}

\providecommand{\Z}{\ensuremath{\mathbb{Z}}}
\providecommand{\N}{\ensuremath{\mathbb{N}}}

\providecommand{\R}{\ensuremath{\mathbb{R}}}
\providecommand{\disUnion}{\ensuremath{\bigsqcup}}
\providecommand{\disunion}{\ensuremath{\sqcup}}
\providecommand{\diam}{\ensuremath{\text{diam}}}
\providecommand{\weak}{\ensuremath{\text{weak}}}
\providecommand{\orbit}{\ensuremath{\text{orbit}}}
\title{Topological Speedups}
\date{\today}
\author{Drew D. Ash}

\begin{document}
\begin{abstract}
Given a dynamical system $(X,T)$ one can define a speedup of $(X,T)$ as another dynamical system conjugate to $S:X\rightarrow X$ where $S(x)=T^{p(x)}(x)$ for some function $p:X\rightarrow\mathbb{Z}^{+}$. In $1985$ Arnoux, Ornstein, and Weiss showed that any aperiodic, not necessarily ergodic, measure preserving system is isomorphic to a speedup of any ergodic measure preserving system. In this paper we study speedups in the topological category. Specifically, we consider minimal homeomorphisms on Cantor spaces. Our main theorem gives conditions on when one such system is a speedup of another. Furthermore, the main theorem serves as a topological analogue of the Arnoux, Ornstein, and Weiss speedup theorem, as well as a ``one-sided" orbit equivalence theorem.
\end{abstract}
\maketitle
\section{Introduction}
In this paper we characterize, in the topological setting, when one minimal Cantor system is a speedup of another. This theorem builds upon two different theorems in dynamics: one theorem from the measure theoretic category, the other from the topological category. Our main theorem is a topological analogue of the speedup theorem of Arnoux, Ornstein, and Weiss \cite{AOW}. Their theorem shows that the realization of a measure preserving system as a speedup of another is very general, however there are restrictions that arise in the topological category. The form of our characterization is very similar to the remarkable theorem of Giordano, Putnam, and Skau \cite[Theorem 2.2]{GPS} in that both theorems the dynamical relations are characterized by associated ordered groups or associated simplices of invariant measures. Whereas in \cite{GPS} they have bijective morphisms from one object onto the other, in our characterization theorem we get surjective and injective morphisms respectively. Furthermore, through the similarity of these theorems we can relate topological speedups to topological orbit equivalence. For example, given a pair of minimal Cantor systems, both of which are uniquely ergodic, if one or both systems is a speedup of the other then the two systems are orbit equivalent. 

These results follow in a long line of results coming from several different research areas of dynamics. The first, and perhaps most general, is that of finding topological analogues for results stemming from ergodic theory. One example, which we will mention a few times throughout this paper, is the topological analogue to the classical ergodic theory result of Dye~\cite{Dye}. Recall Dye's theorem says that any two ergodic transformations on non-atomic Lebesgue probability spaces are orbit equivalent. Over $35$ years later Giordano, Putnam, and Skau gave a complete characterization of when two minimal Cantor systems are orbit equivalent in the topological category. Unlike in the measure theoretic category not all minimal Cantor systems are orbit equivalent in the topological category 

Another line of research we follow is that of speedups themselves. Speedups have mostly been studied in the measurable category. By a \emph{speedup} of a fixed aperiodic measure preserving transformation $(X,\mathscr{B},\mu,T)$ we mean an automorphism of the form  $S(x)=T^{p(x)}(x)$, $p:X\rightarrow\Z^{+}$. One of the earliest people to study speedups-though they were not called this until later- was Neveu in $1969$. He had two papers \cite{Neveu1},\cite{Neveu2}; the latter would eventually give restrictions on what systems can be speedup to each other assuming integrability of $p$. The first major result, after Neveu, came in $1985$ with Arnoux, Ornstein, and Weiss, when they showed: for any ergodic measure preserving transformation $(X,\mathscr{B},\mu,T)$ and any aperiodic, not necessarily ergodic, $(Y,\mathscr{C},\nu, S)$ there is a $\mathscr{B}-$measurable function $p:X\rightarrow\Z^{+}$ such that $\bar{S}(x)=T^{p(x)}(x)$ is invertible $\mu$-a.e. and $(X,\mathscr{B},\mu,\bar{S})$ is isomorphic to $(Y,\mathscr{C},\nu, S)$. Finding a topological analogue to this theorem was the inspiration and impetus for this paper. Interest in measure theoretic speedups has been rekindled as evidenced by the papers by \cite{BBF}, \cite{JM}.

The final line of research our paper follows is that of topological orbit equivalence. Recall that $(X,T)$ and $(Y,S)$ are orbit equivalent if there exists a space isomorphism $F:X\rightarrow Y$ such that for every $x\in X$, $F(\orbit_{T}(x))=\orbit_{S}(F(x))$. Again Dye's theorem says that in the measurable category any two ergodic transformations on non-atomic Lebesgue probability spaces are orbit equivalent. This is not the case in the topological category. In $1995$ Giordano, Putnam, and Skau completely characterized orbit equivalence in the topological category. In doing so, they introduced two new orbit equivalence invariants namely: the dimension group, and having the simplices of invariants measures be affinely isomorphic via a space homeomorphism. We restate their characterization theorem here:

\begin{theorem}\textbf{\cite[Theorem $2.2$]{GPS}:}\label{GPS Theorem} Let $(X_{i},T_{i})$ be Cantor systems $(i=1,2)$. The following are equivalent:
\begin{enumerate}
\item $(X_{1},T_{1})$ and $(X_{2},T_{2})$ are orbit equivalent.
\item The dimension groups $K^{0}(X_{i},T_{i})/Inf(K^{0}(X_{i},T_{i})),\, i=1,2$, are order isomorphic by a map preserving the distinguished order units.
\item There exits a homeomorphism $F:X_{1}\rightarrow X_{2}$ carrying the $T_{1}-$invariant probability measures onto the $T_{2}-$invariant probability measures.
\end{enumerate}
\end{theorem}
Above $K^{0}(X_{i},T_{i})/Inf(K^{0}(X_{i},T_{i}))$ is the group of continuous functions from $X_{i}$ to the integers modulo the subgroup of functions which integrate to 0 against every $T_{i}$-invariant Borel probability measure.

We can view speedups through the lens of orbit equivalence by observing that if $(X_{2},T_{2})$ is a speedup of $(X_{1},T_{1})$ then there exists a homeomorphism $F:X_{1}\rightarrow X_{2}$ such that for every $x\in X_{1}$ we have
$$
F(\orbit_{T_{1}}^{+}(x))\supseteq\orbit_{T_{2}}^{+}(F(x)).
$$
Our main theorem, stated below, has a very similar form to Theorem~\ref{GPS Theorem} above.

\begin{main} Let $(X_{1},T_{1})$ and $(X_{2},T_{2})$ be minimal Cantor systems. The following are equivalent:
\begin{enumerate}
\item $(X_{2},T_{2})$ is a speedup of $(X_{1},T_{1})$
\item There exists
$$
\varphi:K^{0}(X_{2},T_{2})/Inf(K^{0}(X_{2},T_{2}))\twoheadrightarrow K^{0}(X_{1},T_{1})/Inf(K^{0}(X_{1},T_{1}))
$$
a surjective group homomorphism such that $\varphi(K^{0}(X_{2},T_{2})^{+})=K^{0}(X_{1},T_{1})^{+}$ and $\varphi$ preserves the distinguished order units.
\item There exists homeomorphism $F:X\rightarrow Y$, such that $F_{*}:M(X_{1},T_{1})\hookrightarrow M(X_{2},T_{2})$ is an injection.
\end{enumerate}
\end{main}
Here we can see the one-sided and reciprocal nature of our main theorem. Instead of having bijective morphims, as is the case in Giordano, Putnam, and Skau's result, we alternatively have either surjective or injective morphisms from one object to the other: surjective morphism preserving the order unit and taking one positive cone onto the other in the dimension group setting, and an injection, arising from a space homeomorphism, from one simplex of invariant measures to the other. In section $5$ we will prove our main theorem and what's more having a surjective morphism on the dimension groups induces an injective morphism on the simplices of invariant measures (or states associated to the dimension group); hence illustrating the reciprocal nature of speedups. Furthermore, as a consequence of both the Main Theorem and Theorem~\ref{GPS Theorem}, in the case of uniquely ergodic minimal Cantor system speedups characterize orbit equivalence. That is, given two uniquely ergodic minimal Cantor systems if one is a speedup of the other, then the systems are orbit equivalent. In section $6$ of the paper we will define speedup equivalence and show, that speedup equivalence and orbit equivalence are the same in systems with finitely many ergodic measures. Finally, we conclude the paper by presenting an example which shows that speedups can leave the orbit equivalence class of a given minimal transformation. 
\section{Preliminaries}
\subsection{Minimal Cantor systems}
As a general reference for dynamics we recommend: \cite{Walters},\cite{BS},\cite{Petersen}. Throughout this paper $X$ will always be taken to be a Cantor space, that is a compact, metrizable, perfect, zero-dimensional space. A \emph{Cantor system} will consist of a pair $(X,T)$ where $X$ is a Cantor space and $T:X\rightarrow X$ is a homeomorphism. In addition we will require that our homeomorphism be \emph{minimal}, by which we mean that every orbit is dense. Specifically, for every $x$ in $X$ we have that
$$
\overline{\mathcal{O}_{T}(x)}=\overline{\{T^{n}(x):x\in\Z\}}=X
$$
where $\mathcal{O}_{T}(x)$ denotes the orbit of the point $x$. We call such systems $(X,T)$ \emph{minimal Cantor systems}. It is well-known (see \cite{Walters}) that we can replace the density of all full orbits with the density of just the forward orbits. Thus, a homeomorphism $T$ is minimal if for every $x\in X$ we have that
$$
\overline{\mathcal{O}_{T}^{+}(x)}=\overline{\{T^{n}(x):n\in\N\}}=X
$$
where $\mathcal{O}_{T}^{+}(x)$ denotes the forward orbit of the point $x$. 

A helpful example which will be referenced throughout the paper is the dyadic odometer. Here we take $X=\{0,1\}^{\N}$, where $\{0,1\}$ is endowed with the discrete topology, making $X$ into a Cantor space. We define $T$ to be $``+1$ and carry to the right", so for example
$$
.000\dots\overset{T}{\mapsto} .100\dots\overset{T}{\mapsto}.010\dots\overset{T}{\mapsto}.110\dots\overset{T}{\mapsto}001\dots.
$$
Formally, $T$ can be defined as
$$
T(x)(i)
\begin{cases}
0 & \text{ if } i<n \\
1 & \text{ if } i=n \\
x(i) & \text{ if } i>n
\end{cases}
$$
where $n$ is the least positive integer such that $x(n)=0$, and $T$ maps the constantly $1$ sequence to the constantly $0$ sequence. The triadic odometer, which is mentioned later in the paper, is similarly defined on $\{0,1,2\}^{\N}$.

Minimal Cantor systems exhibit a wonderful structure, namely the existence of a refining sequence of Kakutani-Rokhlin tower partitions. These tower partitions, defined below, were instrumental in relating minimal Cantor systems to Bratteli diagrams, and hence dimension groups, AF-Algebras, and many other beautiful results.
\begin{definition}
A \emph{\textbf{Kakutani-Rokhlin tower partition}} of a minimal Cantor system $(X,T)$ is a clopen partition $\mathcal{P}$ of $X$ of the form
$$
\mathcal{P}=\{T^{j}C_{k}:k\in V,\, 0\le j<h_{k}\}
$$

where $V$ is a finite set, $C_{k}$ is a clopen set, and $h_{k}$ is a positive integer. 
\end{definition}
By fixing a $k$ we may refer to a \emph{column} of the partition $\{T^{j}C_{k}:0\le j<h_{k}\}$, and $h_{k}$ is referred to the height of the column. The set $T^{j}C_{k}$ is the $j^{th}$ \emph{level} of the $k^{th}$ column. Furthermore, we refer to 
$$
C=\Union_{k\in V}C_{k}
$$
as the \emph{base} of the Kakutani-Rokhlin tower partition. A visualization of a Kakutani-Rokhlin tower partition  is provided below.
$$
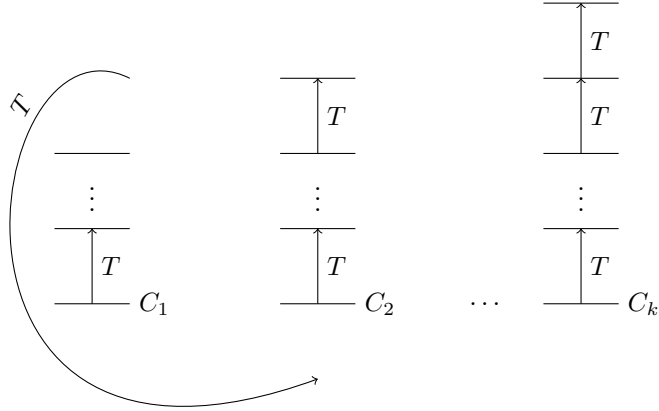
\begin{tikzpicture}
\draw(0,0)--(1,0) node[right]{$C_{1}$};
\draw[->](.5,0)--(.5,1) node[midway,right]{$T$};
\draw(0,1)--(1,1);
\node at (.5,1.5) {$\vdots$};
\draw(0,2.0)--(1,2.0);
\draw(3,0)--(4,0) node[right]{$C_{2}$};
\draw[->](3.5,0)--(3.5,1) node[midway,right]{$T$};
\draw(3,1)--(4,1);
\node at (3.5,1.5) [] {$\vdots$};
\draw(3,2)--(4,2);
\draw[->](3.5,2)--(3.5,3) node[midway,right]{$T$};
\draw(3,3)--(4,3);
\node at (5.75,0) [] {$\dots$};
\draw(6.5,0)--(7.5,0) node[right]{$C_{k}$};
\draw[->](7,0)--(7,1) node[midway,right]{$T$};
\draw(6.5,1)--(7.5,1);
\node at (7,1.5) [] {$\vdots$};
\draw(6.5,2)--(7.5,2);
\draw[->](7,2)--(7,3) node[midway,right]{$T$};
\draw(6.5,3)--(7.5,3);
\draw[->](7,3)--(7,4) node[midway,right]{$T$};
\draw(6.5,4)--(7.5,4);
\draw[->](1.0,3) .. controls (-1,4) and (-2,-3) .. node[near start, sloped, above]{$T$} (3.5,-1);
\end{tikzpicture}
$$
Note $T$ maps the top of each column into the base, and in only special cases does the top of any column map onto the first level of that column. 

A nice property of these Kakutani-Rokhlin tower partitions is that they can have arbitrarily high columns heights and can refine any clopen partition the space. We summarize these properties in the following two propositions.
\begin{proposition}\label{K-R height}
Let $(X,T)$ be a minimal Cantor system, and $n\in\Z^{+}$ be given. There exists a Kakutani-Rokhlin tower partition of $X$,
$$
\{T^{j}(C_{i}):1\le i\le t,\, 0\le j< h_{i} \}
$$
such that for $i=1,2,\dots,t,\, h_{i}>n$.
\end{proposition}
\begin{proposition}\label{K-R refine}
Let $(X,T)$ be a minimal Cantor system, $\mathcal{Q}$ a clopen partition of $X$, and $\mathscr{P}$ a Kakutani-Rokhlin tower partition of $X$. Specifically,
$$
\mathscr{P}=\{T^{j}(C_{i}):1\le i\le t,\, 0\le j<h_{i}\}.
$$
Then we can refine $\mathscr{P}$ into $\mathscr{P}'$ such that $\mathscr{P}'$ refines $\mathcal{Q}$, and $\mathscr{P}'$ maintains its tower structure: that is
$$
\mathscr{P}'=\{T^{j}(C'_{i}):1\le i\le t',\, 0\le j<h'_{i} \}
$$
where $t'$ is the new number of columns and $h'_{i}$ is the new height of the $i^{th}$ column.
\end{proposition}
Putting the following definition and propositions together we get a fundamental theorem not only for this paper, but for the study of minimal Cantor systems in general.
\begin{theorem}\label{K-R}
Let $(X,T)$ be a minimal Cantor system and let $x\in X$. There exists a sequence of Kakutani-Rokhlin tower partitions $(\mathscr{P}(n))_{n\in\N}$ with
$$
\mathscr{P}(n):=\{T^{j}B_{i}(n):1\le i\le t(n),0\le j<h_{i}(n)\}
$$ 
satisfying
\begin{enumerate}
\item $\displaystyle\Intersect_{n\in\N}\Union_{1\le i\le t(n)}B_{i}(n)=\{x\}$ 
\item for every $n$ we have that $\mathscr{P}(n+1)$ is finer than $\mathscr{P}(n)$ i.e. $\mathscr{P}(n)\le\mathscr{P}(n+1)$ for every $n$.
\item $\displaystyle\Union_{n\in\N}\mathscr{P}(n)$ generates the topology of $X$.
\end{enumerate}
\end{theorem}
We will make extensive use of this theorem throughout the proof of the main result of this paper. 
\subsection{Invariant measures associated to minimal Cantor systems}
In this section we will review some standard facts about invariant measures associated to topological dynamical systems and fix notation. Then we will introduce the definition of a dynamical simplex, or $D$-Simplex, which is due to Heidi Dahl, and was inspired by, and extended, the notion of a good measure introduced by Ethan Akin in \cite{Akin}.

First recall that the Bogolioubov-Krylov theorem says that any continuous transformation of a compact metric space has an invariant Borel probability measure. Fix a minimal Cantor system $(X,T)$ and let $M(X)$ denote the collection of all Borel probability measures on $X$. We are interested in the measures in $M(X)$ which are $T$-invariant, and we denote the collection of all $T$-invariant Borel probability measures by $M(X,T)$, i.e.
$$
M(X,T)=\{\mu\in M(X):\mu(T^{-1}(A))=\mu(A)\text{ for every Borel subset $A$}\}
$$
Again by Bogolioubov-Krylov, $M(X,T)\neq\emptyset$. 

The set $M(X,T)$ has a very nice structure as it is a Choquet simplex with respect to the weak$^{*}$ topology; that is, $M(X,T)$ is a compact, convex subset of $M(X)$ in which every measure $\mu$ can be uniquely represented as an integral against a measure $\tau$ which is fully supported on the extreme points, denoted by $\partial_{e}(M(X,T))$. Furthermore recall that a measure $\mu$ is \emph{full or has full support} if $\mu$ gives positive measure to every non-empty open set. Also, we say that a measure $\mu$ is \emph{non-atomic} if $\mu$ gives measure $0$ to singletons. We are now ready to define a $D$-simplex. 
\begin{definition}[Dahl]
Let $K\subseteq M(X)$ be a Choquet simplex consisting of non-atomic probability measures with full support. We say that $K$ is a \emph{\textbf{dynamical simplex (abbreviated $D$-simplex)}} if it satisfies the following two conditions:
\begin{enumerate}
\item For clopen subsets $A$ and $B$ of $X$ with $\mu(A)<\mu(B)$ for all $\mu\in K$, there exists a clopen subset $B_{1}\subseteq B$ such that $\mu(A)=\mu(B_{1})$ for all $\mu\in K$ (this is known as the subset condition).
\item If $\mu,\nu\in\partial_{e}K,\, \mu\neq\nu$, then $\mu$ and $\nu$ are mutually singular, i.e. there exists a measurable set $A\subseteq X$ such that $\mu(A)=1$ and $\nu(A)=0$.
\end{enumerate}
\end{definition}
It is well known that for any minimal Cantor system $(X,T)$, $M(X,T)$ is a Choquet simplex whose extreme points are mutually singular, see \cite[Chapter $6$]{Walters}. The fact that all measures are non-atomic and full both follow from $X$ being uncountable coupled with $T$ being a minimal transformation. Showing $M(X,T)$ is actually a $D$-simplex follows immediately from a proof of Lemma $2.5$ from Glasner and Weiss \cite{Glasner-Weiss}. From this we have the following theorem.
\begin{theorem}
Let $(X,T)$ be a minimal Cantor system. The set $M(X,T)$ is a $D$-simplex.
\end{theorem}
The fact that $M(X,T)$ is a $D$-simplex will play a role in the proof of the main theorem.
\subsection{Ordered groups and dimension groups} One of the more recent tools in the study of minimal Cantor systems, and in particular in the study of topological orbit equivalence is the dimension group. Dimensions groups were first defined by Elliot in \cite{Elliot} using inductive limits of groups. However, the definitions which follow are an equivalent, and more abstract way of defining dimension groups which is due to Effros, Handelman, and Shen \cite{EHS}.

Before we can define what a dimension group is we must first introduce partially ordered groups. A general reference for parially ordered Abelian groups is \cite{Goodearl}, for references specifically related to dynamics we refer the reader to \cite{HPS},\cite{GPS}, and for a summary see \cite{CANT}. 

In this paper we will deal exclusively countable Abelian groups.
\begin{definition} A \emph{\textbf{partially ordered group}} is a countable Abelain group $G$ together with a special subset denoted $G^{+}$, referred to as the \emph{\textbf{positive cone}}, satisfying the following:
\begin{enumerate}
\item $G^{+}+G^{+}\subseteq G^{+}$
\item $G^{+}-G^{+}=G$
\item $G^{+}\intersect(-G^{+})=\{0\}$
\end{enumerate}
\end{definition}
Since we are calling these groups partially ordered given $a,b\in G$ we will write
$$
a\le b\text{ if } b-a\in G^{+}
$$
and we can define a strict inequality, $a<b$ by requesting that $b-a\in G^{+}\backslash\{0\}$. We will further require that our partially ordered Abelian groups be \emph{unperforated} by which we mean: if $a\in G$ and $na\in G^{+}$ for some $n\in\Z^{+}$ then $a\in G^{+}$. We press on towards defining what a dimension group is with the final condition: the Riesz interpolation property.
\begin{definition}
A partially ordered group is said to satisfy the \emph{\textbf{Riesz interpolation property}} if given $a_{1},a_{2},b_{1},b_{2}\in G$ with $a_{i}\le b_{j}$ for $i,j=1,2$, then there exists $c\in G$ such that 
$$
a_{i}\le c\le b_{j} \text{ for }i,j=1,2.
$$
\end{definition}
Finally, we have enough background to define a dimension group.
\begin{definition}
A \emph{\textbf{dimension group}} is an unperforated ordered group $(G,G^{+})$ which satisfies the Riesz interpolation property.
\end{definition}
An example of a dimension group, which will appear multiple times in this paper, is $(\Z[\frac{1}{2}],\Z[\frac{1}{2}]^{+})$ where
$$
\Z\left[\frac{1}{2}\right]=\left\{\frac{a}{2^{b}}:a\in\Z,b\in\N\right\} \text{ and }
\Z\left[\frac{1}{2}\right]^{+}=\left\{x\in\Z\left[\frac{1}{2}\right]:x\ge 0\right\}.
$$
In fact this dimension group is the exact dimension group associated to the dyadic odometer. Furthermore, a theorem by Giordano, Putnam, and Skau, which we will give later in the paper, showed that nearly all dimension groups arise from minimal Cantor systems.

There are two other properties of dimension groups we must discuss before moving forward. The first being the notion of an order unit.
\begin{definition}
Let $(G,G^{+})$ be a partially ordered group, we call $u\in G^{+}$ an \emph{\textbf{order unit}} if for every $a\in G$ there exists an $n\in\N$ such that $a\le nu$. Furthermore, any dimension group with an order unit will be called a \emph{\textbf{unital dimension group.}}
\end{definition}
Note $1$ plays the role of an ordered unit in our example above, which makes $(\Z[\frac{1}{2}],\Z[\frac{1}{2}]^{+},\textbf{1})$ a unital dimension group.

Finally, when dealing with minimal Cantors systems we only encounter \emph{simple dimension groups}, defined below. Seeing as our groups are Abelian, simple does not refer to the group being simple, but rather posits that the order ideal structure is simple.
\begin{definition}
An \emph{\textbf{order ideal}} is a subgroup $J$ so that 
\begin{enumerate}
\item $J=J^{+}-J^{+}$ where $J^{+}=J\intersect G^{+}$
\item if $0\le a\le b\in J$, then $a\in J$
\end{enumerate}
and a dimension group is \emph{\textbf{simple}} if it has no non-trivial order ideals.
\end{definition}

From now on we will only concern ourself with simple dimension groups. There are many connections between dimension groups and minimal Cantor systems and we will highlight some of these connections later in the paper. We need another definition.
\begin{definition}
Let $G$ be a simple dimension group with a fixed order unit $u\in G^{+}\backslash\{0\}$. We say that a homomorphism $p:G\rightarrow\R$ is a \emph{\textbf{state}} if $p$ is positive (i.e. $p(G^{+})\subseteq[0,\infty))$ and $p(u)=1$.
\end{definition}
States play an important role in the order structure of these dimension groups. To see this, let $(G,G^{+},u)$ be a unital simple dimension group (i.e. $(G,G^{+})$ is a simple dimension group and $u$ is an order unit) and let $S_{u}(G)$ denote the collection of all states on $G$. It is known that states always exists and so $S_{u}(G)\neq\emptyset$. Now paraphrasing a result of Effros\cite[Cor. $4.2$]{Effros} we have that
$$
G^{+}=\{a\in G:p(a)>0 \text{ for all } p\in S_{u}(G)\}\union\{0\}
$$
This tells us that by knowing the states we know the order structure of $G$. Furthermore, we can make at least one connection with minimal Cantor systems, which we will make explicit once we have more notation, in that states on the dimension group correspond exactly to invariant measures for the minimal Cantor system associated to this dimension group. Hence there always exists at least one state, just as there always exists at least one invariant measure.

We now would like to single out special elements of any simple dimension group $(G,G^{+})$. First, fix $(G,G^{+})$ a simple unital dimension group with $u\in G^{+}\backslash\{0\}$ an ordered unit. We say that $a\in G$ is an \emph{infinitesimal} if $p(a)=0$ for every $p\in S_{u}(G)$. We will let $Inf(G)$ denote the collection of all infinitesimals of $G$ and we note that it is a subgroup of $G$. Furthermore, if we start with a dimension group $G$ and form the quotient group $G/Inf(G)$, the quotient has a natural order structure coming from $G$ in that $[a]>0$ if $a>0$. From this it can be seen that $G/Inf(G)$ becomes a dimension group in its own right and has no infintesimals other than $[0]$.
\subsection{Dimension groups and dynamical system}
In the section we will give a brief introduction to some basic definitions, notation, and theorems about dimension groups associated to minimal Cantor systems. For a more detailed and motivational exploration of these links we implore the reader to see \cite{GPS},\cite{HPS}. 

Given a minimal Cantor system $(X,T)$, let $C(X,\Z)$ denote the collection of all continuous $\Z$ valued functions on $X$. This is a countable Abelian group under addition. Furthermore, define
$$
K^{0}(X,T)=C(X,\Z)/\{f-f\circ T:f\in C(X,\Z)\}.
$$ 
We denote by $B_{T}=\{f-f\circ T:f\in C(X,\Z)\}$ and call it collection of coboundaries. Define the positive cone, the positive elements, to be
$$
K^{0}(X,T)^{+}=\{[f]:f\ge 0,\, f\in C(X,\Z)\}
$$ 
also let $\textbf{1}$ denote the constantly $1$ function on $X$. We now have the following theorem relating dimension groups arising from minimal Cantor systems.
\begin{theorem}\cite[Theorem $1.12$]{GPS}\label{DimGpReal}
Let $(X,T)$ be a minimal Cantor system. Then $K^{0}(X,T)$ with positive cone $K^{0}(X,T)^{+}$ is a simple, acyclic (i.e. $G\ncong\Z$) dimension group with (canonical) distinguished order unit $\textbf{1}$. Furthermore, if $(G,G^{+})$ is a simple, acyclic dimension group with distinguished order unit $u$, there exists a minimal Cantor system $(X,T)$ so that 
$$
(G,G^{+},u)\cong(K^{0}(X,T),K^{0}(X,T)^{+},\textbf{1})
$$
meaning that there exists an order isomorphism $\alpha:G\rightarrow K^{0}(X,T)$ so that $\alpha(u)=\textbf{1}$.
\end{theorem}
The use of these dimension groups has been used to completely classify minimal Cantor systems up to strong orbit equivalence and orbit equivalence, see \cite{GPS}. The dimension group we concern ourselves with in this paper are dimension groups modulo their infinitesimals. As mentioned previously, there is a lovely connection between states of a dimension group and invariant measures which we will make explicit now. We then can give a simple characterization of the dimension groups that will appear in this paper. First we present the following theorem.
\begin{theorem}\cite[Theorem $1.13$]{GPS}\label{measures-states}
Let $(X,T)$ be a minimal Cantor system. Then
\begin{enumerate}
\item Every $T$-invariant probability measure $\mu$ on $X$ induces a state $T(\mu)$ on $(K^{0}(X,T),K^{0}(X,T)^{+},\textbf{1})$ by $f\rightarrow\int{f}d\mu,\, f\in C(X,\Z)$.
\item The map $T$ is a bijective correspondence between the set of $T$-invariant probability measures on $X$ and the set of states on $(K^{0}(X,T),K^{0}(X,T)^{+},\textbf{1})$.
\end{enumerate}
\end{theorem}
One can verify that this theorem still holds true on $K^{0}(X,T)/Inf(K^{0}(X,T)$. We now have seen states arise as integration against an invariant measure, hence let $Z_{T}=\{f\in C(X,\Z):\int{f}d\mu=0,\, \mu\in M(X,T)\}$, we then have
\begin{center}
$Inf(K^{0}(X,T))=Z_{T}/B_{T}=\{f\in C(X,\Z):\int{f}d\mu=0,\,\mu\in M(X,T)\}/B_{T}$.
\end{center}
Thus,
\begin{center}
$K^{0}(X,T)/Inf(K^{0}(X,T))\cong C(X,\Z)/Z_{T}$
\end{center}
and the order unit $\textbf{1}$ is preserved when $C(X,\Z)/Z_{T}$ is endowed with the induced order of $[f]\ge 0$ if $f\ge 0$ in $C(X,\Z)$.
\section{Speedups}
In this section we will define what we mean by a speedup of a minimal Cantor system $(X,T)$. Furthermore, we explore some of its basic properties which will lead up to the main theorem of the paper. First we define what a speedup is.
\begin{definition} Let $(X_{1},T_{1})$ and $(X_{2},T_{2})$ be minimal Cantor systems. We say $(X_{2},T_{2})$ is a \emph{\textbf{speedup}} of $(X_{1},T_{1})$ if $(X_{2},T_{2})$ is conjugate to $(X,S)$ where $S$ is a minimal homeomorphism of $X$ defined by
$$
S(x)=T_{1}^{p(x)}(x)
$$
where $p:X\rightarrow\Z^{+}$.
\end{definition}
For example if $(X,T)$ is the dyadic odometer, then $(X,T^{3})$ would constitute a speedup of $(X,T)$ as it is again a minimal Cantor system. Now we would like to point out that our definition of speedup is a bit more general in that any minimal Cantor system which is conjugate to $(X,T^{3})$ is also considered to be a speedup of $(X,T)$. We remark that $(X,T^{2})$, or anything conjugate to $(X,T^{2})$, cannot be a speedup of $(X,T)$ as $T^{2}$ is not minimal.

In the paper by Arnoux, Ornstein, and Weiss \cite{AOW}, $p$ is a measurable map. In the topological category we make the observation that if $T^{p(\cdot)}$ is to be continuous then $p$ must be lower semicontinuous.
\begin{proposition} Let $p:X\rightarrow\mathbb{Z}^{+}$ and suppose that $T^{p(x)}(x)=S(x)$ is a minimal Cantor system, then $p$ is lower semicontinuous, hence a Borel map.
\end{proposition}
\begin{proof}
First, we show that for every $n\in\Z^{+}$ we have
$$
p^{-1}(\{n\})\text{ is closed.}
$$
Let $n\in\mathbb{Z}^{+}$, $\{x_{m}\}_{m\ge 1}\subseteq p^{-1}(\{n\})$, and $x\in X$ such that $x_{m}\rightarrow x$; since both $S$ and $T^{n}$ are continuous, we have that
$$
S(x_{m})\rightarrow S(x)\text{ and } T^{n}(x_{m})\rightarrow T^{n}(x).
$$
Since for every $m$, $S(x_{m})=T^{n}(x_{m})$ and by uniqueness of limits we have that
$$
S(x)=T^{n}(x).
$$
We may conclude $p(x)=n$ as a result of $T$ being aperiodic by virtue of being a minimal transformation on a Cantor space.

Recall that a real valued function is \emph{lower semicontinuous} on a topological space if 
$$
\{x\in X:f(x)>\alpha\}
$$
is open for every real $\alpha$. Now let $\alpha\in\mathbb{R}$ be given. Observe that for any $\alpha$ there are only finitely many $n\in\Z^{+}$ such that $n\le\alpha$; thus,
$$
\{x:p(x)\le\alpha\}=\Union_{n\le\alpha}p^{-1}(\{n\})
$$
is a finite union of closed sets whence is closed. Consequently $\{x:p(x)>\alpha\}$ is open, therefore $p$ is lower semicontinuous as desired.
\end{proof}
\begin{remark}
If $p$ is continuous, then $p$ must be bounded as $X$ is compact. However, the converse is true as well. That is, if $p$ is bounded and defines a speedup $S$, then $p$ is continuous. This follows almost immediately from the previous proposition. In this case, where $p$ is bounded, finitely valued, or continuous and $S(x)=T^{p(x)}(x)$ is a speedup of $T$, we call these \emph{\textbf{bounded speedups.}} Bounded speedups are interesting in their own right. For example, by \cite{Neveu2} entropy restrictions arise in what systems can be speedups of others. However, bounded speedups are beyond the scope of this paper.
\end{remark}

One important aspect of speedups is how they interact with the invariant measures of the original system. The following proposition gives the relationship between the invariant measures of the original system and speedups of it. Furthermore, we have an example which shows the relationship below can be strict; thus showing that speedups can leave the conjugacy class of the original system. We'll discuss this more later in the paper. Before we prove this relationship it will be useful to be able to refer to following proposition.
\begin{proposition}\label{M(X,T)=M(X,T^{-1})}
	Suppose $(X,T)$ is a minimal Cantor system, then $M(X,T)=M(X,T^{-1})$.
\end{proposition}
We now show how speedups interact with the invariant measures of the original system.
\begin{proposition}\label{Simplex inequality} Let $(X,T)$ be a minimal Cantor system. If $(X,S)$ is a speedup of $(X,T)$ then $M(X,T)\subseteq M(X,S)$.
\end{proposition}
\begin{proof}
Let $p:X\rightarrow\mathbb{Z}^{+}$ be such that $S(x)=T^{p(x)}(x)$ is a minimal homeomorophism of $X$ and let $\mu\in M(X,T)$. Observe by Proposition~\ref{M(X,T)=M(X,T^{-1})} it suffices to simply show that $\mu\in M(X,S^{-1})$. Let $A\in\mathscr{B}(X),$ we then have
\begin{align*}
\mu(S(A))&=\mu\left(S\left(\displaystyle\bigsqcup_{n\in\mathbb{Z}^{+}}A\intersect p^{-1}(\{n\})\right)\right) \\
&=\mu\left(\displaystyle\bigsqcup_{n\in\mathbb{Z}^{+}}S(A\intersect p^{-1}(\{n\}))\right) \\
&=\mu\left(\disUnion_{n\in\Z^{+}}T^{n}(A\intersect p^{-1}(\{n\})) \right)\\
&=\displaystyle\sum_{n\in\mathbb{Z}^{+}}\mu(T^{n}(A\intersect p^{-1}(\{n\}))) \\
&=\displaystyle\sum_{n\in\mathbb{Z}^{+}}\mu(A\intersect p^{-1}(\{n\}))\text{ as $\mu\in M(X,T)$.} \\
&=\mu(A)
\end{align*}
\end{proof}
Notice that this proposition gives us an immediate restriction on when one system can be a speedup of another. For example, the previous proposition rules out the possibility of the triadic odometer being a speedup of the dyadic odometer, and vice versa, as both systems are uniquely ergodic and do not share the same clopen value set. The natural question to ask is: is this the only such restriction? We answer this and more with the statement of the main theorem of the paper.

\begin{theorem}\label{Main Theorem} Let $(X_{1},T_{1})$ and $(X_{2},T_{2})$ be minimal Cantor systems and let  
$$
G_{1}=C(X_{1},\mathbb{Z})/Z_{T_{1}} \text{ and } G_{2}=C(X_{2},\mathbb{Z})/Z_{T_{2}}.
$$
Where $Z_{T_{i}}=\{g\in C(X,\Z):\int{g}d\mu=0\, \forall\mu\in M(X_{i},T_{i})\}$.

The following are equivalent:
\begin{enumerate}
\item $(X_{2},T_{2})$ is a speedup of $(X_{1},T_{1})$.
\item There exists 
$$
\varphi:(G_{2},G_{2}^{+},\textbf{1})\twoheadrightarrow (G_{1},G_{1}^{+},\textbf{1})
$$
a surjective group homomorophism such that $\varphi(G_{2}^{+})=G_{1}^{+}$ and $\varphi(\textbf{1})=\textbf{1}$.
\item There exists homeomorphism $F:X\rightarrow Y$, such that $F_{*}:M(X_{1},T_{1})\hookrightarrow M(X_{2},T_{2})$ is an injection.
\end{enumerate}
\end{theorem}
We will break up the proof of the main theorem into three sections, as each part of the proof requires a different set of lemmas. The main difficulty is proving $(3)$ implies $(1)$.
\subsection{Proof of $(1)$ implies $(2)$}
\begin{proof}
Since $(X_{2},T_{2})$ is a speedup of $(X_{1},T_{1})$, $(X_{2},T_{2})$ is conjugate, through a conjugacy $\mathpzc{k}$, to $(X,S)$ where $S:X\rightarrow X$ 
$$
S(x)=T_{1}^{p(x)}(x)
$$
and $p:X\rightarrow\Z^{+}$. Let $H_{1}=C(X,\Z)/Z_{S}$ and $(H_{1},H_{1}^{+},\textbf{1})$ be the unital dimension group associated to $(X,S)$. Hence, right composition of $\mathpzc{k}$ induces a unital dimension group isomorphism $\varphi_{1}:(H_{1},H_{1}^{+},\textbf{1})\rightarrow(G_{2},G_{2}^{+},\textbf{1})$. Define $\varphi_{2}:(G_{2},G_{2}^{+},\textbf{1})\rightarrow(G_{1},G_{1}^{+},\textbf{1})$ by
$$
\varphi_{2}([g]_{S})=[g]_{T_{1}}. 
$$
Observe, Proposition~\ref{Simplex inequality} gives us 
$$
Z_{S}\subseteq Z_{T_{1}}
$$
whence $\varphi_{2}$ is well defined. It is standard to check that $\varphi_{2}$ is a surjective group homomorphism(see the third isomorphism theorem for groups). Moreover, as right composition by $\mathpzc{k}$ doesn't affect positivity of elements, nor does it alter the order unit. One can verify
$$
\varphi_{2}(G_{2}^{+})=G_{1}^{+}\text{ and }\varphi_{2}(\textbf{1})=\textbf{1}.
$$
Therefore $\varphi=\varphi_{2}\circ\varphi_{1}$ is our desired group homomorphism.
\end{proof}
\subsection{Proof of $(2)$ implies $(3)$}
In order to proceed from $(2)$ to $(3)$ we would like to make use of \cite[Thm $2.2$]{GPS}. To do so we will need to extend the first isomorphism theorem from groups to partially ordered Abelian groups with interpolation. We recall for the reader one of the main theorems from \cite{GPS}.

\begin{theorem}\label{GPS}\textbf{\cite[Theorem $2.2$]{GPS}:} Let $(X_{i},T_{i})$ be Cantor systems $(i=1,2)$. The following are equivalent:
\begin{enumerate}[(i)]
\item $(X_{1},T_{1})$ and $(X_{2},T_{2})$ are orbit equivalent.
\item The dimension groups $K^{0}(X_{i},T_{i})/Inf(K^{0}(X_{i},T_{i})),\, i=1,2$, are order isomorphic by a map preserving the distinguished order units.
\item There exits a homeomorphism $F:X_{1}\rightarrow X_{2}$ carrying the $T_{1}-$invariant probability measures onto the $T_{2}-$invariant probability measures.
\end{enumerate}
\end{theorem}

Furthermore, recall what an isomorphism is in the category of unital partially ordered Abelian groups with interpolation.

\begin{definition} An \emph{\textbf{isomorphism}} between two unital partially ordered Abelian groups say $(G,G^{+},u)$ and $(H,H^{+},v)$ is a map $\varphi:G\rightarrow H$ a group and order  isomorphism and $\varphi(u)=v.$ In such a case we say that $(G,G^{+},u)$ is isomorphic to $(H,H^{+},v)$, written $(G,G^{+},u)\cong(H,H^{+},v).$
\end{definition}

We now proceed with a short proof of the first isomorphism theorem in the category of partially ordered Abelian groups with interpolation.
\begin{theorem}\label{first isomorphism} Let $(G,G^{+},u)$ and $(H,H^{+},v)$ be unital dimension groups. If $\varphi:H\rightarrow G$ is a surjective, order and order unit preserving homomorphism with $\varphi(H^{+})=G^{+}$, then 
$$(H/\ker(\varphi), H^{+}/\ker(\varphi),[v])\cong (G,G^{+},u)
$$ 
as unital dimension groups.
\end{theorem}
\begin{proof}
Define $\hat{\varphi}:H/\ker(\varphi)\rightarrow G$ by 
$$
\hat{\varphi}([h])=\varphi(h)
$$
for $h\in H$. By the first isomorphism theorem for groups $\hat{\varphi}$ is a group isomorphism; thus it suffices to show that $\hat{\varphi}\left(H^{+}/\ker(\varphi)\right)=G^{+}$, and $\hat{\varphi}([v])=u$. These follow immediately as $\varphi(H^{+})=G^{+}$ and $\varphi(v)=u$.
\end{proof}
We will need one more proposition before tackling $(2)\Rightarrow (3)$ and it begins to illustrate the reciprocal nature of the main theorem.
\begin{proposition}\label{3.9}
Let $\varphi:G_{2}\twoheadrightarrow G_{1}$ be as in $(2)$ of Theorem $\ref{Main Theorem}$. Then there exists an injection $\varphi_{*}:M(X_{1},T_{1})\hookrightarrow M(X_{2},T_{2})$
\end{proposition}
\begin{proof}
We will show that $\varphi$ induces an injective map on $M(X_{1},T_{1})$ into the state space of $G_{2}$. From there we appeal to Theorem $\ref{measures-states}$, which says that the states and invariant measures are in bijective correspondence. Composing these two functions gives us our injection from $M(X_{1},T_{1})$ into $M(X_{2},T_{2})$.

Let $\mu\in M(X_{1},T_{1})$, $h\in C(Y,\Z)$ and define
\begin{align*}
\varphi_{*}\mu[h]&=\displaystyle\int_{X}\varphi([h])\,d\mu \\
&=\displaystyle\int_{X}g\, d\mu \text{\hspace{5mm} where $g\in C(X,\Z)$ and $g\in\varphi([h])$}
\end{align*}
Let us first show that $\varphi_{*}$ is well-defined. Let $h\in C(Y,\Z)$ and $g_{1},g_{2}\in C(X,\Z)$ be such that $g_{1},g_{2}\in\varphi([h])$; thus there exists $i\in Inf(G)$ such that $g_{1}+i=g_{2}$. Now we calculate
\begin{align*}
\displaystyle\int_{X}g_{2}\, d\mu &= \displaystyle\int_{X}(g_{1}+i)\, d\mu \\
&=\displaystyle\int_{X}g_{1}\, d\mu
\end{align*}
so $\varphi_{*}$ is well-defined. Since $\varphi$ is order unit preserving we see that
$$
\varphi_{*}\mu[1]=\displaystyle\int_{X}1\, d\mu =1.
$$
To see that $\varphi_{*}\mu$ is positive, let $h\in C(Y,\Z)$ be such that for every $x$, $h(x)\ge 0$, thus $[h]\in H^{+},$ and whence $\varphi([h])\ge 0$ as $\varphi$ is positive. So there exists $g\in C(X,\Z)$ such that for every $x,$ $g(x)\ge 0$ and $g\in\varphi([h])$. Thus,
$$
\varphi_{*}\mu[h]=\displaystyle\int_{X}g\,d\mu\ge 0.
$$
Finally, to see that $\varphi_{*}\mu$ is a homomorphism, let $h_{1},h_{2}\in C(Y,\Z)$. Observe,
\begin{align*}
\varphi_{*}\mu[h_{1}+h_{2}]&=\displaystyle\int_{X}\varphi([h_{1}+h_{2}])\, d\mu \\
&=\displaystyle\int_{X}\varphi([h_{1}]+[h_{2}])\,d\mu \\
&=\displaystyle\int_{X}(\varphi([h_{1}])+\varphi([h_{2}]))\,d\mu \\
&=\displaystyle\int_{X}\varphi([h_{1}])\,d\mu+\displaystyle\int_{X}\varphi([h_{2}])\,d\mu \\
&=\varphi_{*}\mu[h_{1}]+\varphi_{*}\mu[h_{2}].
\end{align*}
Therefore, $\varphi_{*}\mu$ is a state on $G_{2}$ as desired. 

Now we will show that $\varphi_{*}$ is injective. Let $\mu,\nu\in M(X_{1},T_{1})$ such that $\mu\neq\nu$. So there exists a clopen set $C$ such that,
$$
\int_{X}\mathbbm{1}_{C}\,d\mu=\mu(C)\neq\nu(C)=\int_{X}\mathbbm{1}_{C}\,d\nu
$$
Since $\varphi(H^{+})=G^{+}$ there exists $h\in C(Y,\Z)$,\text{ for every $x$,} $h(x)\ge 0$ such that $\varphi([h])=[\mathbbm{1}_{C}]$, rather $\mathbbm{1}_{C}\in\varphi([h])$. Now we compute,
$$
\varphi_{*}\mu([h])=\int_{X}\mathbbm{1}_{C}\,d\mu=\mu(C)\neq\nu(C)=\int_{X}\mathbbm{1}_{C}\,d\nu=\varphi_{*}\nu([h]).
$$
So $\varphi_{*}$ is injective. Recall \cite[Cor. $4.2$]{Effros} which says that the set of states is in bijective correspondence with the set of invariant measures and so we get our desired injection, by composing $\varphi_{*}$ with this bijection. 
\end{proof}

With Theorem~\ref{first isomorphism}, Proposition~\ref{3.9} and Theorem $2.2$ of \cite{GPS} at our disposal, we wish to dispense of $(2)\Rightarrow (3)$. 
\begin{proof}
By assuming $(2)$ and in conjunction with Theorem~\ref{first isomorphism} we know that $\hat{\varphi}$ is an unital dimension group isomorphism
$$
\hat{\varphi}:(G_{2}/\ker(\varphi),G_{2}^{+}/\ker(\varphi),[\textbf{1}]_{\varphi})\rightarrow (G_{1},G_{1}^{+},\textbf{1});
$$
so in particular $(G_{2}/\ker(\varphi),G_{2}^{+}/\ker(\varphi),[\textbf{1}]_{\varphi})$ is itself a unital dimension group. As a result of the isomorphism, $(G_{2}/\ker(\varphi),G_{2}^{+}/\ker(\varphi),[\textbf{1}]_{\varphi})$ must have one infinitesimal, namely $[0]_{\varphi}$. Furthermore, $H^{+}_{1}/\ker(\varphi)$ is determined by $\varphi_{*}(M(X_{1},T_{1}))$ by Proposition~\ref{3.9}.  By \cite[Thm. 2.2]{GPS} there exists a homeomorphism $F:X_{1}\rightarrow X_{2}$ such that the invariant measures associated to $(X_{1},T_{1})$ are taken bijectively onto the $g$-invariant measures, where $g$ is a minimal realization of
$$
(G_{2}/\ker(\varphi),G_{2}^{+}/\ker(\varphi),[\textbf{1}]_{\varphi})
$$
by Theorem~\ref{DimGpReal}. Finally, Proposition~\ref{3.9} also shows that the invariant measures associated to $G_{2}/\ker(\varphi)$ are a subset of $M(X_{2},T_{2})$, and we have our injection from $M(X_{1},T_{1})$ into $M(X_{2},T_{2})$ via a space homeomorphism from $X_{1}$ to $X_{2}$ as desired. Note, that $(X_{2},g)$ and $(X_{1},T_{1})$ are orbit equivalent as a result \cite[Thm. 2.2]{GPS}, since their dimension groups modulo infinitesimals are isomorphic as dimension groups.
\end{proof}

\subsection{Proof of $(3)$ implies $(1)$}
This is by far the most technical portion of the paper. The idea of the proof is quite similar to the construction presented in the Arnoux, Ornstein, and Weiss paper \cite{AOW}. In fact our key lemma, Lemma $\ref{Key Lemma}$ is a topological version of the key lemma from \cite{AOW} and a modification of Proposition $2.6$ from \cite{Glasner-Weiss}. Note a key difference in our lemma is the range of our $p$ map is $\Z^{+}$ instead of $\Z$. This lemma allows us to actually construct the speedup on the non-final levels on a Kakutani-Rokhlin tower partition. 

Before moving forward with the construction to prove $(3)$ implies $(1)$ we will prove a short sequence of lemmas culminating with our key lemma, Lemma \ref{Key Lemma}. Again, many of the following propositions and lemmas are similar to propositions and lemmas found in \cite{Glasner-Weiss}.
\begin{proposition}\label{clopen}
Let $(X,T)$ be a minimal Cantor system. We have for every $\varepsilon>0$ there exists a nonempty clopen set $C$ such that for all $\mu\in M(X,T)$, $\mu(C)<\varepsilon$.
\end{proposition}
\begin{proof}
Suppose, towards a contradiction, there exists an $\varepsilon>0$ such that for every non-empty clopen set $C$ there exists $\nu\in M(X,T)$ such that $\nu(C)\ge\varepsilon$. Fix $x\in X$, let $C_{n}=B_{\frac{1}{n}}(x)$ and let $\mu_{n}$ be a measure in $M(X,T)$ such that $\mu_{n}(C_{n})\ge\varepsilon$. By compactness of $M(X,T)$ there exists $\nu\in M(X,T)$ and $n_{k}\nearrow\infty$ such that 
$$
\begin{tikzpicture}
\matrix(m)[matrix of math nodes,
row sep=2.6em, column sep=2.8em,
text height=1.5ex, text depth=0.25ex]
{\mu_{n_{k}} & \nu. \\};
\path[->,font=\scriptsize,>=angle 90]
(m-1-1) edge node[auto] {$\weak^{*}$} (m-1-2);
\end{tikzpicture}
$$
Clearly,
$$
\displaystyle\Intersect_{n=1}^{\infty}C_{n}=\displaystyle\Intersect_{k=1}^{\infty}C_{n_{k}}=\{x\}
$$
and so
$$
\nu(\{x\})=\displaystyle\lim_{k\rightarrow\infty}\nu(C_{n_{k}})
$$
and we claim that for all $k\in\mathbb{Z}^{+},\,\nu(C_{n_{k}})\ge\varepsilon$. Fix $k\in\mathbb{Z}^{+}$, since $C_{n_{k}}$ is clopen we have by definition
$$
\nu(C_{n_{k}})=\lim_{j\rightarrow\infty}\mu_{n_{j}}(C_{n_{k}})
$$
and for $k<j$ we have that 
\begin{align*}
C_{n_{k}}\supseteq C_{n_{j}}&\Rightarrow \mu_{n_{j}}(C_{n_{k}})\ge\mu_{n_{j}}(C_{n_{j}})\ge\varepsilon
\end{align*}
thus $\nu(C_{n_{k}})\ge\varepsilon$. So we see that
$$
\nu(\{x\})=\displaystyle\lim_{k\rightarrow\infty}\nu(C_{n_{k}})\ge\varepsilon>0
$$
contradicting the fact that $\nu$ must be non-atomic.
\end{proof}
We immediately use this proposition to prove the following lemma. 
\begin{lemma}\label{clopen1}
Let $(X,T)$ be a minimal Cantor system. Then for every $\varepsilon>0$ there exists $\delta>0$ such that for every $A\in\mathscr{B}(X)$ with $\emph{\diam}(A)<\delta$ and every $\mu\in M(X,T),$ we have $\mu(A)<\varepsilon$.
\end{lemma}
\begin{proof}
Let $\varepsilon>0$ be given, by Proposition $\ref{clopen}$ there exists a non-empty clopen set $C$ such that for all $\mu\in M(X,T),\,\, 0<\mu(C)<\varepsilon$. Since $C$ is non-empty, clopen, and as $T$ is minimal there exists $N\in\mathbb{Z}^{+}$ such that
$$
X=\Union_{i=-N}^{N}T^{i}(C).
$$
Let $\delta>0$ be the Lebesgue number for the open cover $\{T^{i}C\}_{i=-N}^{N}$ (recall that a Lebesgue number for an open covering $\mathcal{A}$ of a compact metric space $X$ is a constant $\delta>0$ such that for each subset of $X$ having diameter less than $\delta$, there exists an element of $\mathcal{A}$ containing it). Now let $A\in\mathscr{B}(X)$ with $\diam(A)<\delta$, then
\begin{align*}
\diam(A)<\delta&\Rightarrow A\subseteq T^{i}(C) \text{\hspace{5mm} for some $i\in\{-N,\dots,N\}$} \\
&\Rightarrow\mu(A)\le\mu(T^{i}(C)) \text{\hspace{5mm} for every $\mu\in M(X,T)$} \\
&\Rightarrow\mu(A)\le\mu(C) \text{\hspace{5mm} as $\mu\in M(X,T)$} \\
&\Rightarrow\mu(A)<\varepsilon.
\end{align*}
So for every $\mu\in M(X,T)$ and $A\in\mathscr{B}(X)$ with $\diam(A)<\delta$ we have $\mu(A)<\varepsilon$ as desired.
\end{proof}
Before we can state and prove one of our key lemmas we need one more proposition.
\begin{proposition}\label{tall tower}
Let $(X,T)$ be a minimal Cantor system, and $f:X\rightarrow\R$ a continuous function. If 
$$
\inf\left\{\int_{X}f\,d\mu:\mu\in M(X,T)\right\}>c>0
$$
then there exists a $N_{0}\in\N$ such that for every $n\ge N_{0}$ and for all $x\in X$ we have
$$
\dfrac{1}{n}\sum_{j=0}^{n-1}f(T^{j}(x))\ge c.
$$
\end{proposition}
\begin{proof}
Fix $f\in C(X,\R)$ and suppose, towards a contradiction, that our proposition is false; that is, there is no such $N_{0}\in\N$. So there exists $\{N_{k}\}_{k\ge0}$ and $\{x_{k}\}_{k\ge0}$ such that $N_{k}\nearrow\infty$, and for a fixed $k$
$$
\dfrac{1}{N_{k}}\sum_{j=0}^{N_{k}-1}f(T^{j}(x_{k}))<c.
$$
Consider the following sequence of measures $\{\mu_{k}\}_{k\ge 0}$, where for fixed $k$ we have
$$
\mu_{k}=\dfrac{1}{N_{k}}\sum_{j=0}^{N_{k}-1}\delta_{T^{j}(x_{k})}
$$
where $\delta$ represents the Dirac measure. By compactness of $M(X)$, the collection of all Borel probability measures on $X$, there exists $\nu\in M(X)$ and increasing sequence $\{k_{\ell}\}_{l\ge 0}\nearrow\infty$ such that
$$
\begin{tikzpicture}
\matrix(m)[matrix of math nodes,
row sep=2.6em, column sep=2.8em,
text height=1.5ex, text depth=0.25ex]
{\mu_{k_{\ell}} & \nu \\};
\path[->,font=\scriptsize,>=angle 90]
(m-1-1) edge node[auto] {$\weak^{*}$} (m-1-2);
\end{tikzpicture}
$$
Recall by \cite[Theorem $6.9$]{Walters} $\nu\in M(X,T)$; we will now show that $\int_{X}f\,d\nu\le c$ which will give us our contradiction. Since $f$ is continuous we have that
\begin{align*}
\displaystyle\int_{X}f\,d\nu&=\displaystyle\lim_{\ell\rightarrow\infty}\int_{X}f\,d\mu_{k_{\ell}} \\
&=\displaystyle\lim_{\ell\rightarrow\infty}\dfrac{1}{N_{\ell}}\sum_{j=0}^{N_{k_{\ell}}-1}f(T^{j}(x_{k_{\ell}})) \\
&\le c.
\end{align*}
This is a contradiction, which proves our proposition. 
\end{proof}
We use Propositions~\ref{clopen} and \ref{tall tower} in conjunction with Lemma~\ref{clopen1} to prove Lemma~\ref{pkey lemma}. This lemma serves as a precursor to the key lemma, and is instrumental for proving Lemma~\ref{Key Lemma}.
\begin{lemma}\label{pkey lemma}
Let $(X,T)$ be a minimal Cantor system, and let $A,B$ be non-empty, disjoint, clopen subsets of $X$. If for all $\mu\in M(X,T),\, \mu(A)<\mu(B),$ then there exists $p:A\rightarrow\Z^{+}$ such that $S:A\rightarrow B$ defined as $S(x)=T^{p(x)}(x)$ is a homeomorphism onto its image.
\end{lemma}
\begin{proof}
Let $A,B$ be disjoint clopen subsets of $X$ and define $f=\mathbbm{1}_{B}-\mathbbm{1}_{A}$. Since both $A$ and $B$ are clopen it follows that $f:X\rightarrow\Z$ is continuous. Moreover since $\int{f}d\mu>0$ for every $\mu\in M(X,T)$ and $M(X,T)$ is compact in the $\weak^{*}$ topology it follows by assumption that 
$$
\inf\left\{\int_{X}f\, d\mu:\mu\in M(X,T)\right\}>0.
$$
Choose $c\in\mathbb{R}$ so that 
$$
\inf\left\{\int_{X}f\, d\mu:\mu\in M(X,T)\right\}>c>0
$$
So by Proposition~\ref{tall tower} find $N_{0}$ large such that for every $n\ge N_{0}$ and every $x\in X$ we have
$$
\dfrac{1}{n}\displaystyle\sum_{j=0}^{n-1}f(T^{j}x)\ge c.
$$
Use Proposition~\ref{K-R height} to construct a Kakutani-Rokhlin tower partition of $X$ such that for each $i,\, h_{i}\ge N_{0}$. Let the following denote our tall Kakuntani-Rokhlin tower partition:
$$
\{T^{j}(D_{i}):1\le i\le t,\, 0\le j<h_{i}\}.
$$ 
Use Proposition~\ref{K-R refine} to refine each tower with respect to the partition $\{A,B,(A\union B)^{c}\}$. By a slight abuse of notation we will not rename our new Kakutani-Rokhlin tower partition, and with that let us look at a single column of our partition. Fix $i=1$, and consider the column
$$
\left\{T^{j}(D_{1}):0\le j<h_{1}\right\}.
$$
Let $x\in D_{1}$, then as $h_{1}\ge N_{0}$ we must have that 
\begin{equation}
\dfrac{1}{h_{1}}\displaystyle\sum_{j=0}^{h_{1}-1}f(T^{j}x)\ge c>0
\end{equation}
thus there are more $B$ levels than $A$ levels in this column. In other words let $J$ and $K$ be defined below
\begin{align*}
J&=\{j_{1},j_{2},\dots,j_{m}:T^{j_{i}}(D_{1})\intersect A\neq\emptyset,\, i=1,2,\dots m\} \\
K&=\{k_{1},k_{2},\dots,k_{r}:T^{k_{i}}(D_{1})\intersect B\neq\emptyset, i=1,2,\dots,r\}
\end{align*}
and by $(1)$ we have that $|J|<|K|$. Choose any injection $\Gamma:J\hookrightarrow K$. 

We exploit the inherit order structure of the column to define our map $p.$ First, we give a picture with an arbitrary injection to help the reader visualize what is going on. All $A$-levels in our first column are colored \textcolor{red}{red} and all of the $B$-levels in the first column are colored \textcolor{blue}{blue}.
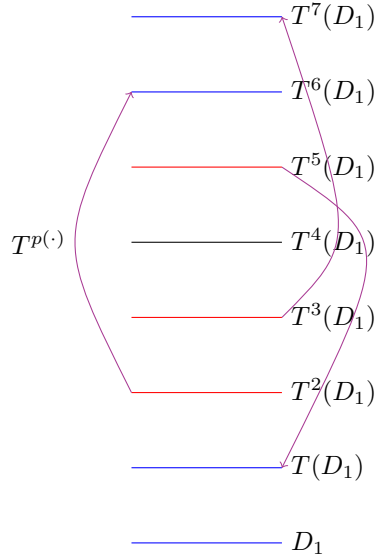
$$
\begin{tikzpicture}
\draw[blue](0,0)--(2,0) node[black][right]{$D_{1}$};
\draw[blue](0,1)--(2,1) node[black][right]{$T(D_{1})$};
\draw[red](0,2)--(2,2) node[black][right]{$T^{2}(D_{1})$};
\draw[red](0,3)--(2,3) node[black][right]{$T^{3}(D_{1})$};
\draw(0,4)--(2,4) node[black][right]{$T^{4}(D_{1})$};
\draw[red](0,5)--(2,5) node[black][right]{$T^{5}(D_{1})$};
\draw[blue](0,6)--(2,6) node[black][right]{$T^{6}(D_{1})$};
\draw[blue](0,7)--(2,7) node[black][right]{$T^{7}(D_{1})$};
\draw[Mulberry][->](0,2).. controls (-1,4) .. node[black][midway,left]{$T^{p(\cdot)}$} (0,6);
\draw[Mulberry][->](2,3).. controls (3,4) .. (2,7);
\draw[Mulberry][->](2,5) .. controls (3.5,4) ..(2,1);
\end{tikzpicture}
$$
We break the definition of $T^{p(\cdot)}$ into the following two cases. First fix $i\in\{1,2,\dots,m\}$.
\vskip 5mm
\textbf{Case $1$: $\Gamma(j_{i})>j_{i}$.}
In this case we can simply define $p:T^{j_{i}}(D_{1})\rightarrow\Z^{+}$ by $p(x)=\Gamma(j_{i})-j_{i}$. By assumption $p$ is positive and as $T$ is a homeomorphism we have
$$
T^{p(\cdot)}=T^{\Gamma(j_{i})-j_{i}}
$$
is a homeomorphism from $T^{j_{i}}(D_{1})\subseteq A$ to $T^{\Gamma(j_{i})}(D_{1})\subseteq B$. Furthermore, we see that 
$$
T^{p(\cdot)}(T^{j_{i}}(D_{1}))=T^{\Gamma({j_{i}})-j_{i}}(T^{j_{i}}D_{1})=T^{\Gamma(j_{i})}(D_{1})\subseteq B.
$$
Note $T^{p(\cdot)}$ simply moves $x$ up the requisite number of levels in the tower as $T^{\Gamma(j_{i})}(D_{1})$ lies above $T^{j_{1}}(D_{1})$ in the column by assumption. This finishes the first case.

\textbf{Case $2$: $\Gamma(j_{i})<j_{i}$.}
In this case we see that we must map an $A$ level into a $B$ level which is below it in our column. In this case we can't move down the tower as $p$ must be positively valued. To this end let $T^{\lambda(\cdot)}:T^{j_{i}}(D_{1})\rightarrow T^{j_{i}}(D_{1})$ be the first return map where recall, 
$$
\lambda(x)=\inf\{n>0:T^{n}x\in T^{j_{i}}(D_{1})\}.
$$
The map $\lambda$ is well defined by virtue of $T^{j_{1}}D_{1}$ being clopen and $T$ minimal. Moreover, one can see that $\lambda$ is continuous, hence $\lambda$ is finitely valued as $T^{j_{1}}(D_{1})$ is compact. Furthermore, it is well known that $T^{\lambda}:T^{j_{1}}(D_{1})\rightarrow T^{j_{1}}(D_{1})$ is a homeomorphism; if we let $S=T^{\Gamma(j_{i})-j_{i}}\circ T^{\lambda}$ we have that $S:T^{j_{i}}(D_{1})\rightarrow T^{\Gamma(j_{i})}$ is a homeomorphism and the resulting $p$ function on $T^{j_{i}}(D_{1})$ is
$$
p(x)=\lambda(x)-(\Gamma(j_{i})-j_{i}).
$$
Thus, all that is left to show is that $p$ is a positive function. However, let 
$$
\lambda(T^{j_{i}}(D_{1}))=\{t_{1},t_{2},\dots,t_{n}\}.
$$
Observe points must traverse the tower in a specified order, thus for each $\ell\in\{1,2,\dots, n\}$ we must have that $t_{\ell}\ge h_{1}$, hence for each $\ell,\, t_{\ell}-(\Gamma(j_{i})-j_{i})>0$. Therefore, we have found our $S:T^{j_{i}}(D_{1})\rightarrow T^{\Gamma(j_{i})}(D_{1})$ of the form $S(x)=T^{p(x)}(x)$, where $p:T^{j_{i}}(D_{1})\rightarrow\Z^{+}$ as desired.

Continuing for each $i$, and then for each column we see that we define $p$ on all of $A$. Furthermore, it is clear that $T^{p(\cdot)}$ is a continuous surjection from $A$ onto its image in $B$, as $T^{p(\cdot)}$ is a homeomorphism on each level of $A$. To see that $T^{p(\cdot)}$ is injective, hence a homeomorphism, observe that $T^{p(\cdot)}$ is a homeomorphism when restricted to any $A$ level in any column in the Kakutani-Rokhlin tower partition. Moreover, the $T^{p(\cdot)}$ image of any two distinct, hence disjoint, $A$ levels is again disjoint. Finally, as all columns of the Kakutani-Rokhlin tower partition are disjoint $T^{p(\cdot)}$ maintains its injectivity and is therefore a homeomorphisms from $A$ onto its image in $B$.
\end{proof}
We now immediately use Lemma~\ref{pkey lemma} to prove the final lemma needed in order to prove our key lemma.
\begin{lemma}\label{induction}
	Let $(X,T)$ be a minimal Cantor system, and let $A,B\subseteq X$ be non-empty, disjoint, clopen subsets of $X$ with
	$$
	\mu(A)=\mu(B)
	$$ 
	for every $\mu\in M(X,T)$. Moreover, fix $x\in A$, $y\in B$ and let $\varepsilon>0$ be given. Then there exists clopen sets $A_{1}\subseteq A,\, B_{1}\subseteq B$ with the following properties:
	\begin{enumerate}
		\item $x\in A_{1}$ and $y\in B_{1}$
		\item $\diam(A_{1})<\varepsilon,\, \diam(B_{1})<\varepsilon$
		\item For every $\mu\in M(X,T),\,\mu(A_{1})=\mu(B_{1})$, $ \mu(A_{1})<\dfrac{\mu(A)}{2}$, $\mu(B_{1})<\dfrac{\mu(B)}{2}$
		\item There exists $p:A\backslash A_{1}\rightarrow\Z^{+}$ such that $T^{p(\cdot)}:A\backslash A_{1}\rightarrow B\backslash B_{1}$ is a homeomorphism.
	\end{enumerate}
\end{lemma}
\begin{proof}
Let $A$ and $B$ be non-empty, disjoint, clopen subsets of $X$, and fix $x\in A$ and $y\in B$, and let $\varepsilon>0$ be given. Recall that every measure $\mu\in M(X,T)$ is full, i.e. gives positive measure to non-empty open sets, whence $\int\mathbbm{1}_{A}d\mu>0$. Let
$$
\alpha=\inf\left\{\int_{X}\mathbbm{1}_{A}\,d\mu:\mu\in M(X,T)\right\}.
$$
Since for every $\mu\in M(X,T)$, $\mu(A)=\mu(B)$ we also have that
$$
\alpha=\inf\left\{\int_{X}\mathbbm{1}_{B}\,d\mu:\mu\in M(X,T)\right\}.
$$
Observe, $\mathbbm{1}_{A}$ is continuous as $A$ is clopen and since $M(X,T)$ is compact in the weak$^{*}$ topology the above infimum is achieved; whence $\alpha>0$.
By Lemma $\ref{clopen1}$ there exists $\delta_{\alpha}>0$ and such that for every $K\in\mathscr{B}(X)$ and for every $\mu\in M(X,T)$
\begin{center}
$\diam(K)<\delta_{\alpha}\Rightarrow\mu(K)<\dfrac{\alpha}{2}$.
\end{center}
Find clopen set $A_{\frac{1}{3}}\subsetneqq A,$ such that
$$
x\in A_{\frac{1}{3}} \text{ and } \diam(A_{\frac{1}{3}})<\min\{\delta_{\alpha},\varepsilon\}.
$$
Let
$$
\varepsilon_{1}=\inf\left\{\int_{X}\mathbbm{1}_{A_{\frac{1}{3}}}\,d\mu:\mu\in M(X,T)\right\}>0
$$
and use Lemma~\ref{clopen1} to obtain $\delta_{1}>0$ such that for every $K\in\mathscr{B}(X)$ and for every $\mu\in M(X,T)$
$$
\diam(K)<\delta_{1}\Rightarrow\mu(K)<\varepsilon_{1}.
$$
Find clopen subset $B_{1}\subseteq B$ such that
$$
y\in B_{1}\text{ and }\diam(B_{1})<\min\{\varepsilon,\delta_{1},\delta_{\alpha}\};
$$
thus we have for all $\mu\in M(X,T)$ we have 
\begin{align*}
\mu(B_{1})<\varepsilon_{1}<\mu(A_{\frac{1}{3}})\Rightarrow\mu(A\backslash A_{\frac{1}{3}})<\mu(B\backslash B_{1}).
\end{align*}
Apply Lemma $\ref{pkey lemma}$ to $T$ and get $p_{1}:A\backslash A_{\frac{1}{3}}\rightarrow\Z^{+}$ such that $S:A\backslash A_{\frac{1}{3}}\rightarrow B\backslash B_{1}$, defined by $S(x)=T^{p_{1}(x)}(x)$, is a homeomorphism onto its image. Then, $B\backslash S(A\backslash A_{\frac{1}{3}})=B_{1}\disunion\ U_{1}$ where $U_{1}$ is a non-empty clopen set and $B_{1}$ and $U_{1}$ are disjoint. Furthermore, for every $\mu\in M(X,T)$ we have that
\begin{equation}
\mu(A_{\frac{1}{3}})=\mu(B_{1})+\mu(U_{1})
\end{equation}
We can visualize this as below.
$$
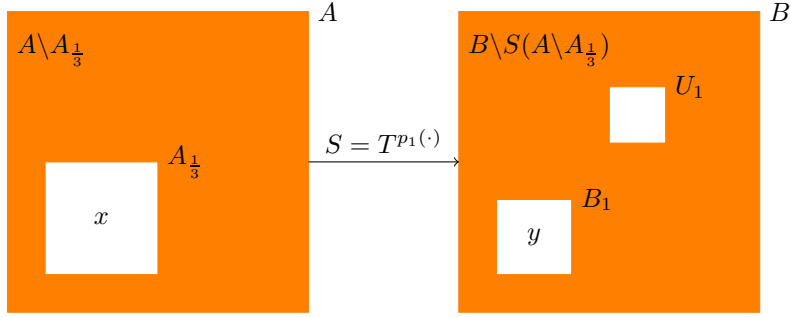
\begin{tikzpicture}
\filldraw[orange, even odd rule] 
(0,0) rectangle (4,4) node[black][right]{$A$} (.5,.5) rectangle (2,2) node[black][midway]{$x$};
\node at (2,2)[black][right]{$A_{\frac{1}{3}}$};
\filldraw[orange, even odd rule]
(6,0) rectangle (10,4) node[black][right]{$B$} (6.5,.5) rectangle (7.5,1.5) node[black][midway]{$y$} (8,3) rectangle (8.75,2.25);
\node at (7.5,1.5) [black][right]{$B_{1}$};
\node at (8.75,3)[black][right]{$U_{1}$};
\draw[->] (4,2)--(6,2) node[midway,above]{$S=T^{p_{1}(\cdot)}$};
\node at (0,3.5)[right]{$A\backslash A_{\frac{1}{3}}$};
\node at (6,3.5)[right]{$B\backslash S(A\backslash A_{\frac{1}{3}})$};
\end{tikzpicture}
$$
Here is the intertwining nature of the proof; in order to extend to $S$ to more of $A$, we apply Lemma~\ref{pkey lemma} to $T^{-1}$ with respect to the clopen sets $U_{1}$ and $A_{\frac{1}{3}}$ with a small neighborhood of $x$ removed. By $(2)$ above we have for every $\mu\in M(X,T)$
$$
\mu(U_{1})<\mu(A_{\frac{1}{3}})
$$
and let
$$
\varepsilon_{2}=\inf\left\{\int_{X}(\mathbbm{1}_{A_{\frac{1}{3}}}-\mathbbm{1}_{U_{1}})d\mu:\mu\in M(X,T)\right\}>0.
$$
By Lemma $\ref{clopen1}$ there exists $\delta_{2}>0$ such that for every $\mu\in M(X,T)$ and every $K\in\mathscr{B}(X)$ we have,
$$
\diam(K)<\delta_{2}\Rightarrow \mu(K)<\varepsilon_{2}.
$$
Find clopen set $A_{\frac{2}{3}}\subsetneqq A_{\frac{1}{3}}$ such that 
$$
x\in A_{\frac{2}{3}}\text{ and } \diam(A_{\frac{2}{3}})<d_{2}=\min\left\{\delta_{2},\diam(A_{\frac{1}{3}})\right\}.
$$
Thus for all $\mu\in M(X,T)$ we have that
\begin{align*}
\mu(A_{\frac{1}{3}}\backslash A_{\frac{2}{3}})&=\mu(A_{\frac{1}{3}})-\mu(A_{\frac{2}{3}}) \\
&>\mu(A_{\frac{1}{3}})-(\mu(A_{\frac{1}{3}})-\mu(U_{1})) \\
&=\mu(U_{1}).
\end{align*}
Applying Lemma $\ref{pkey lemma}$ to $T^{-1}$ and $U_{1}$, recall by Proposition $\ref{M(X,T)=M(X,T^{-1})}$ we have $M(X,T)=M(X,T^{-1})$, we get $\hat{p}_{2}:U_{1}\rightarrow\Z^{+}$ such that$ (T^{-1})^{\hat{p}_{2}(\cdot)}:U_{1}\rightarrow A_{\frac{1}{3}}\backslash A_{\frac{2}{3}}$ is a homeomorphsim onto its image in $A_{\frac{1}{3}}\backslash A_{\frac{2}{3}}$.
$$
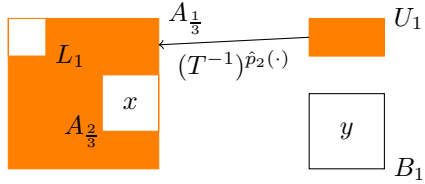
\begin{tikzpicture}
\filldraw[orange, even odd rule] (0,0) rectangle (2,2)
(1.25,1.25) rectangle (2,.5) node[midway,black]{$x$} (0,2) rectangle (.5,1.5) node[black,right]{$L_{1}$};
\node at (1.35,.5) [black,left]{$A_{\frac{2}{3}}$};
\draw (4,0) rectangle (5,1) node[midway,black]{$y$};
\node at (5,0) [black,right]{$B_{1}$};
\node at (2,2) [black,right]{$A_{\frac{1}{3}}$};
\filldraw[orange] (4,1.5) rectangle (5,2) node[right,black]{$U_{1}$};
\draw[->] (4,1.75) -- (2,1.65) node[midway,below]{$(T^{-1})^{\hat{p}_{2}(\cdot)}$};
\end{tikzpicture}
$$
We now use $\hat{p_{2}}$ to define $p_{2}:(T^{-1})^{\hat{p_{2}}}(U_{1})\rightarrow\Z^{+}$ by
$$
p_{2}((T^{-1})^{\hat{p}_{2}(z)}(z))=\hat{p}_{2}(z).
$$
Observe, for any $z\in\ U_{1}$ we have
$$
T^{p_{2}(z)}(T^{-p_{2}(z)}(z))=z
$$ 
and similarly the reverse composition is the identity, whence $T^{p_{2}(\cdot)}$ is not only a bijection, but the inverse function to $(T^{-1})^{\hat{p}_{2}(\cdot)}$, and so is a homeomorphism itself.
	
This intertwining allows us to map more of $A$ onto $B$ using only positive powers of $T$ and also to ensure that the diameter of $B_{1}$ is small. As was the case with $p_{1}$ we see that 
$$
A_{\frac{1}{3}}\backslash T^{-p_{2}(\cdot)}(U_{1})=A_{\frac{2}{3}}\disunion L_{1}
$$
where $L_{1}$ is a clopen subset of $A_{\frac{1}{3}}$ with $A_{\frac{2}{3}}$ and $L_{1}$ being disjoint. Again we have the following equality for every $\mu\in M(X,T)$
$$
\mu(B_{1})=\mu(A_{\frac{2}{3}})+\mu(L_{1}).
$$
Thus, by defining
$$
A_{1}=A_{\frac{2}{3}}\disunion L_{1}
$$
we have $A_{1}$ and $B_{1}$ as desired.
\end{proof}
We will use induction on our previous lemma to prove our key Lemma.
\begin{lemma}\label{Key Lemma}
Let $(X,T)$ be a minimal Cantor system and let $A,B$ be non-empty disjoint, clopen subsets of $X$. If for all $\mu\in M(X,T),\, \mu(A)=\mu(B)$, then there exists $p:A\rightarrow\Z^{+}$ such that $S:A\rightarrow B$, defined as $S(x)=T^{p(x)}(x)$, is a homeomorphism onto $B$.
\end{lemma}
\begin{proof}
	Let $A$ and $B$ be non-empty, disjoint, clopen subsets of $X$ and $x\in A$. Since $T$ is minimal there exists $n\in\Z^{+}$ such that $T^{n}(x)\in B$, let $y=T^{n}(x)$. We will use induction to find a decreasing sequences of sets $\{A_{n}\}_{n\ge 0}$ and $\{B_{n}\}_{n\ge 0}$ such that
	$$
	\Intersect_{n\ge 0}A_{n}=\{x\} \text{ and }\Intersect_{n\ge 0}B_{n}=\{y\}
	$$
	all while defining $S$ on larger and larger parts of $A$. Let $\varepsilon_{1}=\min\{\diam(A),\diam(B), 1\}$, then using Lemma~\ref{induction} find clopen subsets $A_{1}$ and $B_{1}$ such that
	\begin{enumerate}
		\item $x\in A_{1}, y\in B_{1}$
		\item $\diam(A_{1})<\varepsilon_{1},\diam(B_{1})<\varepsilon_{1}$
		\item for every $\mu\in M(X,T),\, \mu(A_{1})=\mu(B_{1})$ and $\mu(A_{1})<\dfrac{\mu(A)}{2},\, \mu(B_{1})<\dfrac{\mu(B)}{2}$
		\item Find $p_{1}:A\backslash A_{1}\rightarrow\Z^{+}$ such that
		$$
		S_{1}=T^{p_{1}(\cdot)}:A\backslash A_{1}\rightarrow B\backslash B_{1}
		$$
		is a homeomorphism.
	\end{enumerate} 
	Now having defined $A_{n}\subseteq A_{n-1}$ and $B_{n}\subseteq B_{n-1}$ with $x\in A_{n},\, y\in B_{n}$ and $\diam(A_{n})<\varepsilon_{n},\, \diam(B_{n})<\varepsilon_{n}$ where 
	$$
	\varepsilon_{n}=\min\left\{\diam(A_{n-1}),\diam(B_{n-1}),\frac{1}{n}\right\}.
	$$
	Moreover, we also have for all $\mu\in M(X,T),$
	$$
	\mu(A_{n})=\mu(B_{n}) \text{ and } \mu(A_{n})<\dfrac{\mu(A_{n-1})}{2},\, \mu(B_{n-1})<\dfrac{B_{n-1}}{2} 
	$$
	and $p_{n}:A_{n-1}\backslash A_{n}\rightarrow\Z^{+}$ such that
	$$
	S_{n}:A_{n-1}\backslash A_{n}\rightarrow B_{n-1}\backslash B_{n}
	$$
	is a homeomorphism. Use Lemma~\ref{induction} with $\varepsilon_{n+1}=\min\{\diam(A_{n}),\diam(B_{n}),\frac{1}{n+1}\}$ to find clopen sets $A_{n+1}$ and $B_{n+1}$ such that 
	\begin{enumerate}
		\item $x\in A_{n+1}, y\in B_{n+1}$
		\item $\diam(A_{n+1})<\varepsilon_{n+1},\diam(B_{n+1})<\varepsilon_{n+1}$
		\item For every $\mu\in M(X,T),\, \mu(A_{n+1})=\mu(B_{n+1})$ and $\mu(A_{n+1})<\dfrac{\mu(A_{n})}{2},\, \mu(B_{n+1})<\dfrac{\mu(B_{n})}{2}$
		\item Find $p_{n+1}:A_{n}\backslash A_{n+1}\rightarrow\Z^{+}$ such that
		$$
		S_{n+1}=T^{p_{n+1}(\cdot)}:A_{n}\backslash A_{n+1}\rightarrow B_{n}\backslash B_{n+1}
		$$
		is a homeomorphism.
	\end{enumerate}
Therefore, by induction we have defined $p:A\backslash\{x\}$ by taking $$p(x)=p_{n}(x)$$ where $x\in A_{n}\backslash A_{n+1}$.
Moreover, we observe at this point $T^{p(\cdot)}:A\backslash\{x\}\rightarrow B\backslash\{y\}$ is a homeomorphism. We extend $p$ to all of $A$ by defining $p(x)=n$. Consequently $T^{p(\cdot)}$ is a bijection on $A$.

All that is left to show is that $T^{p(\cdot)}$ is continuous on $A$. By construction $T^{p(\cdot)}$ is continuous at all points in $A$ less our exceptional point $x$. Let $\varepsilon>0$ be given. Then by the construction there exists an $n$ such that $B_{n}\subseteq B_{\varepsilon}(y)$; thus $T^{p(\cdot)}(A_{n+1})\subseteq B_{n}$, as $T^{p(x)}(x)=y$. Hence taking $\delta>0$ such that the ball of radius $\delta$ about $x$, $B_{\delta}(x)\subseteq A_{n+1}$ we have that $T^{p(\cdot)}$ is continuous at $x$, whence is continuous on all of $A$. Therefore, we have defined $p$ in such a way that the map
$$
T^{p(\cdot)}:A\rightarrow B
$$
is a homeomorphism as desired.
\end{proof}
An immediate corollary of this lemma, in conjunction with Proposition~\ref{Simplex inequality} is the following lemma. 
\begin{lemma}\label{cpartition}
Let $(X,T)$ be a minimal Cantor system and $A,B\subseteq X$ clopen subsets such that $A\intersect B=\emptyset$. If for all $\mu\in M(X,T),\, \mu(A)=\mu(B)$, then for any clopen partition of $A$, say $A=\disUnion_{i=1}^{n}A_{i}$, there exists clopen sets $B_{i}\subseteq B$ with $B=\disUnion_{i=1}^{n}B_{i}$ such that for all $\mu\in M(X,T)$ and for each $i$ we have
$$
\mu(A_{i})=\mu(B_{i})
$$
\end{lemma}
We will use this lemma in the proof of the main theorem which is soon to follow.
We will make use of the following definition due to Dahl.
\begin{definition}[H. Dahl]
Let $K\subseteq M(X)$, where $X$ is a Cantor set, be a Choquet simplex consisting of non-atomic, Borel, probability measures. We say that $K$ is a \textbf{dynamical simplex (D-simplex)} if it satisfies the following two conditions:
\begin{enumerate}
\item For clopen subsets $A$ and $B$ of $X$ with $\mu(A)<\mu(B)$ for all $\mu\in K$, there exists a clopen subset $B_{1}\subseteq B$ such that $\mu(A)=\mu(B_{1})$ for all $\mu\in K$.
\item If $\nu,\mu\in\partial_{e}K,\, \nu\neq\mu,$ then $\mu$ and $\nu$ are mutually singular.
\end{enumerate}
\end{definition}
It is well know that condition $(2)$ is satisfied by every $M(X,T)$ for any continuous map on a compact metric space $X$. Furthermore, thanks to Glasner and Weiss we have the following theorem. 
\begin{theorem}\cite[Lemma $2.5$]{Glasner-Weiss}\label{D-Simplex} Let $(X,T)$ be a minimal Cantor system and $M(X,T)$ its associated Choquet simplex of $T-$invariant measures. Then $M(X,T)$ is a $D-$simplex.
\end{theorem}

Theorem $\ref{D-Simplex}$ becomes useful in construction of the speedup which proves $(3)\Rightarrow (1)$. We have enough background to finish the proof of the main theorem. We recall the final portion of the main theorem we have left to prove.
We will show that  

\begin{theorem}Given $(X_{1},T_{1})$ and $(X_{2},T_{2})$ minimal Cantor system. If there exists a homeomorphism $F:X_{1}\rightarrow X_{2}$ such that $F_{*}:M(X_{1},T_{1})\hookrightarrow M(X_{2},T_{2})$ is an injection, then $(X_{2},T_{2})$ is a speedup of $(X_{1},T_{1})$.
\end{theorem}
\begin{proof}
We begin with a sketch of the proof to keep in mind. The idea of the construction is to take a refining sequence of Kakutani-Rokhlin tower partitions in $X_{2}$ and copy them in $X_{1}$ using the homeomorphism $F^{-1}$. We observe for any fixed tower in $X_{2}$ its copy in $X_{1}$ has the property that all levels in this tower have the same measure for every $T_{1}-$invariant measure. Now using Lemma~\ref{Key Lemma} we can define the speedup on all non-final levels of the tower. Then we define a set conjugacy from one tower to another. We simply iterate this process refining each previous tower. We have a great deal of freedom in this construction, enough to ensure the base and tops of the towers converge to prespecified singletons, say $x$ and $T_{1}^{-1}x$, and that the sequence of towers generates the topology on $X_{1}$. 

We begin by fixing $x_{0}\in X_{1}$ and let $\{A_{n}\}_{n\ge 0}$ be a nested sequence of clopen sets, where
$$
\Intersect_{n\ge 0}A_{n}=\{x_{0}\}.
$$
For each $n$ let $Z_{n}=T_{1}^{-1}(A_{n})$ and so,
$$
\Intersect_{n\ge 0}Z_{n}=\{T_{1}^{-1}x_{0}\}.
$$
We may assume with no loss of generality that $A_{0}\intersect Z_{0}=\emptyset$. Moreover, observe for every $\mu\in M(X_{1},T_{1})$ and every $n$ $\mu(A_{n})=\mu(Z_{n})$. That being said, let $\alpha_{0}$ be defined below, 
$$
\alpha_{0}=\min\left\{\int_{X}\mathbbm{1}_{A_{0}}\,d\mu:\mu\in M(X_{1},T_{1})\right\}.
$$
Coupling the fact that $M(X_{1},T_{1})$ is compact in the weak$^{*}$ topology and both $A_{0}$ is clopen, we may conclude $\alpha_{0}>0$: let $\varepsilon_{0}=\alpha_{0}$. Apply Theorem~\ref{K-R} to create $\{\mathcal{Q}(n)\}_{n\ge 0}$, a sequence of Kakutani-Rohklin tower partitions of $X_{2}$. Specifically,
$$
\mathcal{Q}(n)=\{T_{2}^{j}(B_{i}(n)):1\le i\le t(n),0\le j<h_{i}(n)\}
$$
where $t(n)$ represents the total number of columns and $h_{i}(n)$ represents the height of the $i^{th}$ column in the $n^{th}$ Kakutani-Rokhlin tower partition of $X_{2}$. Furthermore, $\{\mathcal{Q}(n)\}_{n\ge 0}$ has the following three properties:
\begin{enumerate}
\item $\displaystyle\Intersect_{n\in\N}\left(\Union_{1\le i\le t(n)}B_{i}(n)\right)=\{y\}$
\item For every $n$ we have $\mathcal{Q}(n+1)$ is finer than $\mathcal{Q}(n)$.
\item $\Union_{n\in\N}\mathcal{Q}(n)$ generates the topology of $X_{2}$.
\end{enumerate}
Let $\{\mathcal{P}(n)\}_{n\ge 0}$ be a sequence of finite clopen partitions which generates the topology on $X$. Use Lemma~\ref{clopen1} with respect to $\varepsilon_{0}$ and obtain a $\delta_{0}>0$ such that for every $K\in\mathscr{B}(X_{2})$ with diam$(K)<\delta_{0}$ we have for every $\nu\in M(X_{2},T_{2}),\, \nu(K)<\varepsilon_{0}$. Since
$$
\displaystyle\Intersect_{n\ge 0}\left(\Union_{1\le i\le t(n)}B_{i}(n)\right)=\{y\}
$$
there exists an $n_{0}$ such that 
$$
\diam\left(\Union_{1\le i\le t(n_{0})}B_{i}(n_{0})\right)<\delta_{0}
$$
thus, for every $\nu\in M(X_{2},T_{2})$ we have that
$$
\nu\left(\Union_{1\le i\le t(n_{0})}B_{i}(n_{0})\right)<\varepsilon_{0}.
$$
Below we give a picture of $(X_{2},T_{2})$ partitioned into $\mathcal{Q}(n_{0})$. We will use $F^{-1}$ to copy this tower partition into $X_{1}$.
\vskip 2mm
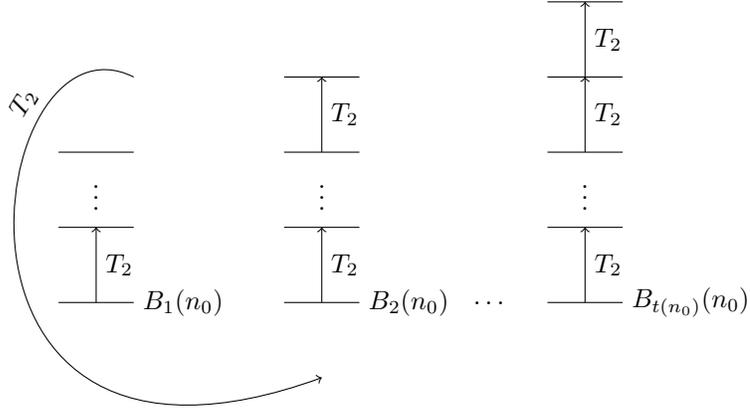
$$
\begin{tikzpicture}
\draw(0,0)--(1,0) node[right]{$B_{1}(n_{0})$};
\draw[->](.5,0)--(.5,1) node[midway,right]{$T_{2}$};
\draw(0,1)--(1,1);
\node at (.5,1.5) {$\vdots$};
\draw(0,2.0)--(1,2.0);
\draw(3,0)--(4,0) node[right]{$B_{2}(n_{0})$};
\draw[->](3.5,0)--(3.5,1) node[midway,right]{$T_{2}$};
\draw(3,1)--(4,1);
\node at (3.5,1.5) [] {$\vdots$};
\draw(3,2)--(4,2);
\draw[->](3.5,2)--(3.5,3) node[midway,right]{$T_{2}$};
\draw(3,3)--(4,3);
\node at (5.75,0) [] {$\dots$};
\draw(6.5,0)--(7.5,0) node[right]{$B_{t(n_{0})}(n_{0})$};
\draw[->](7,0)--(7,1) node[midway,right]{$T_{2}$};
\draw(6.5,1)--(7.5,1);
\node at (7,1.5) [] {$\vdots$};
\draw(6.5,2)--(7.5,2);
\draw[->](7,2)--(7,3) node[midway,right]{$T_{2}$};
\draw(6.5,3)--(7.5,3);
\draw[->](7,3)--(7,4) node[midway,right]{$T_{2}$};
\draw(6.5,4)--(7.5,4);
\draw[->](1.0,3) .. controls (-1,4) and (-2,-3) .. node[near start, sloped, above]{$T_{2}$} (3.5,-1);
\end{tikzpicture}
$$
\vskip 2mm 

Define for $1\le i\le t(n_{0})$ and $0\le j<h_{i}(n_{0})$
$$
C'(i,j)=F^{-1}(T_{2}^{j}(B_{i}(n_{0}))).
$$
We will make a series of alterations to each $C'(i,j)$ resulting in $C(i,j)$ with 
$$
\mu(C'(i,j))=\mu(C(i,j))
$$
for all $\mu\in M(X_{1},T_{1})$. Furthermore, this will be done iteratively and once completed we will have the following
$$
x\in\Union_{i=1}^{t(n_{0})}C(i,0)\subseteq A_{0}\quad\text{ and }\quad T_{1}^{-1}x\in\Union_{i=1}^{t(n_{0})} C(i,h_{i}(n_{0})-1)\subseteq Z_{0}.
$$
\textbf{First ensure $C(i,0)\subseteq  A_{0}$ for each $i$.}

Recall that for every $\mu\in M(X_{1},T_{1})$
$$
\mu\left(\disUnion_{i=1}^{t(n_{0})}C'(i,0)\right)=\sum_{i=1}^{t(n_{0})}\mu(C'(i,0))=\sum_{i=1}^{t(n_{0})}\mu(F^{-1}B_{i}(n_{0}))=\sum_{i=1}^{t(n_{0})}\nu_{\mu}(B_{i}(n_{0}))<\varepsilon_{0}\le\mu(A_{0}),
$$
in particular for every $\mu\in M(X_{1},T_{1})$,
\begin{equation}
\mu(A_{0})-\sum_{i=1}^{t(n_{0})}\mu(C'(i,0))>0.
\end{equation}
Define
$$
D_{0}(i)=C'(i,0)\intersect A_{0}\text{ and }D'_{0}(i)=C'(i,0)\intersect A_{0}^{c}
$$
and by the above we have for every $\mu\in M(X_{1},T_{1})$
\begin{equation}
\sum_{i=1}^{t(n_{0})}\mu(D'_{0}(i))<\mu\left(A_{0}\backslash\disUnion_{i=1}^{t(n_{0})}D_{0}(i)\right).
\end{equation}
Fix $i=1$. It may be the case that $D'_{0}(1)\neq\emptyset$ and in this case we wish to amend this, and to do it in a way which preserves all the measures of each clopen set $C'(i,j)$. 
We know from $(4)$ above that for $\mu\in M(X_{1},T_{1})$
$$
\mu(D'_{0}(1))<\mu\left(A_{0}\backslash\disUnion_{i=1}^{t(n_{0})}D_{0}(i)\right).
$$
Thus, as $M(X_{1},T_{1})$ is a D-simplex, there exists $C_{1}\subseteq A_{0}\backslash\disUnion_{i=1}^{t(n_{0})}D_{0}(i)$ clopen such that for every $\mu\in M(X_{1},T_{1}),\, \mu(C_{1})=\mu(D'_{0}(1)).$ Note, $C_{1}$ is partitioned by
$$
\disUnion_{i=1}^{t(n_{0})}\disUnion_{j=0}^{h_{i}(n_{0})-1}C'(i,j)
$$
into
$$
C_{1}=\disUnion_{k=1}^{m}C_{i_{k},j_{k}}(1),
$$
where $C_{i_{k},j_{k}}(1)\subseteq C'(i_{k},j_{k})$. Hence, by Lemma~\ref{cpartition} there exists a partition of $D'_{0}(1)$,
$$
D'_{0}(1)=\disUnion_{k=1}^{m}D_{i_{k},j_{k}}(1)
$$
where for all $\mu\in M(X_{1},T_{1})$ and each $k=1,2,\dots,m$ 
$$
\mu(C_{i_{k},j_{k}}(1))=\mu(D_{i_{k},j_{k}}(1)).
$$
Define 
$$
C(1,0)=D_{0}(1)\disunion C_{1}
$$
and for each $k=1,2,\dots,m$
$$
C''(i_{k},j_{k})=C'(i_{k},j_{k})\backslash C_{i_{k},j_{k}}(1)\disunion D_{i_{k},j_{k}}(1).
$$
Observe, for every $\mu\in M(X_{1},T_{1})$ and $k=1,2,\dots,m$ we have
$$
\mu(C(1,0))=\mu(C'(1,0))\quad \mu(C''(i_{k},j_{k}))=\mu(C'(i_{k},j_{k})),
$$
and of course all the measure of the unaffected $C'(i,j)$ still have the same measure for each $\mu\in M(X_{1},T_{1})$. Combining $(3)$ and $(4)$ from above reveals
$$
\sum_{i=2}^{t(n_{0})}\mu(D'_{0}(i))<\mu\left(A_{0}\backslash\disUnion_{i=2}^{t(n_{0})}D_{0}(i)\disunion C(1,0)\right).
$$
We now simply repeat the above argument. Inequalities $(3)$ and $(4)$ allow us to do this construction for each $i=1,2,\dots,t(n_{0})$ defining $C(i,0)$ for $i=1,2,\dots,t(n_{0})$. Furthermore, by construction we have the following two properties
\begin{enumerate}
\item for every $\mu\in M(X_{1},T_{1})$ and every $i$, $\mu(C(i,0))=\mu(C'(i,0))$.
\item for every $i\neq j$, $C(i,0)\intersect C(j,0)=\emptyset$.
\end{enumerate}

\textbf{Second, ensure $x\in\Union_{i=1}^{t(n_{0})}C(i,0)$.}

In all of the adjusting to construct $C(i,0)$, $i=1,2,\dots, t(n_{0}),$ we may not have captured $x$. If not, then 
$$
x\in A_{0}\backslash\disUnion_{i=1}^{n_{0}}C(i,0).
$$
Use Proposition~\ref{clopen1} and find a small enough clopen subset of $A_{0}\backslash\disUnion_{i=1}^{n_{0}}C(i,0)$ containing $x$ and exchange it with part of $C(1,0)$.

\textbf{Third, repeat steps one and two to obtain $C(i,h_{i}(n_{0})-1),\, i=1,2,\dots,t(n_{0})$.}

Notice that $\varepsilon_{0}=\min\{\alpha_{0},\zeta_{0}\}$, so we can repeat the above two steps using the same sort of calculations and construction to obtain clopen sets $C(i,h_{i}(n_{0})-1),\, i=1,2,\dots,t(n_{0})$ such that
$$
T_{1}^{-1}x\in\disUnion_{i=1}^{t(n_{0})}C(i,h_{i}(n_{0})-1)\subseteq Z_{0}.
$$
Furthermore, for every $\mu\in M(X_{1},T_{1})$
$$
\mu(C'(i,h_{i}(n_{0})-1))=\mu(C(i,h_{i}(n_{0})-1)).
$$
Having defined $C(i,0)$ and $C(i,h_{i}(n_{0})-1)$ for $i=1,2,\dots,m$ we wish to keep consistent notation and thus rename any $C'(i,j)$ to simply $C(i,j).$
Whence, we have the following: 
\begin{itemize}
\item $x\in\displaystyle\disUnion_{i=1}^{t(n_{0})}C(i,0)\subseteq A_{0}$
\item $T_{1}^{-1}x\in\displaystyle\disUnion_{i=1}^{t(n_{0})}C(i,h_{i}(n_{0})-1)\subseteq Z_{0}$
\item For every $\mu\in M(X_{1},T_{1})$ and fixed $i=1,2,\dots,t(n_{0})$ we have
$$
\mu(C(i,0))=\mu(C(i,1))=\dots=\mu(C(i,h_{i}(n_{0})-1))
$$
\end{itemize}
Now we use repeated applications of our key lemma, Lemma~\ref{Key Lemma}, to define our speedup on nearly all of $X_{1}$. Specifically,
$$
S(C(i,j))=C(i,j+1)
$$
for each $1\le i\le t(n_{0})$ and $0\le j<h_{i}(n_{0})-1$. Thus $S$ is defined on 

$X_{1}\backslash\left(\disUnion_{i=1}^{t(n_{0})}C(i,h_{i}(n_{0})-1)\right)$.
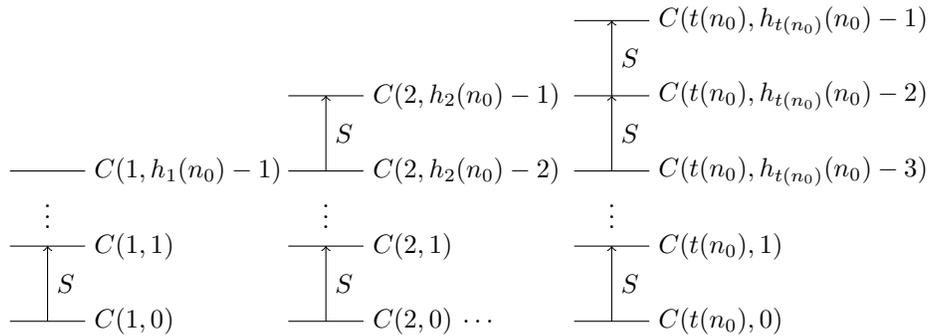
$$
\begin{tikzpicture} 
\draw(-.5,0)--(.5,0) node[right]{$C(1,0)$};
\draw[->](0,0)--(0,1) node[midway,right]{$S$};
\draw(-.5,1)--(.5,1) node[right]{$C(1,1)$};
\node at (0,1.5) {$\vdots$};
\draw(-.5,2.0)--(.5,2.0) node[right]{$C(1,h_{1}(n_{0})-1)$};
\draw(3.2,0)--(4.2,0) node[right]{$C(2,0)$};
\draw[->](3.7,0)--(3.7,1) node[midway,right]{$S$};
\draw(3.2,1)--(4.2,1) node[right]{$C(2,1)$};
\node at (3.7,1.5) [] {$\vdots$};
\draw(3.2,2)--(4.2,2) node[right]{$C(2,h_{2}(n_{0})-2)$};
\draw[->](3.7,2)--(3.7,3) node[midway,right]{$S$};
\draw(3.2,3)--(4.2,3) node[right]{$C(2,h_{2}(n_{0})-1)$};
\node at (5.75,0) [] {$\dots$};
\draw(7,0)--(8,0) node[right]{$C(t(n_{0}),0)$};
\draw[->](7.5,0)--(7.5,1) node[midway,right]{$S$};
\draw(7,1)--(8,1) node[right]{$C(t(n_{0}),1)$};
\node at (7.5,1.5) [] {$\vdots$};
\draw(7,2)--(8,2) node[right]{$C(t(n_{0}),h_{t(n_{0})}(n_{0})-3)$};
\draw[->](7.5,2)--(7.5,3) node[midway,right]{$S$};
\draw(7,3)--(8,3) node[right]{$C(t(n_{0}),h_{t(n_{0})}(n_{0})-2)$};
\draw[->](7.5,3)--(7.5,4) node[midway,right]{$S$};
\draw(7,4)--(8,4) node[right]{$C(t(n_{0}),h_{t(n_{0})}(n_{0})-1)$};
\end{tikzpicture}
$$
\vskip 2mm 

Formally, let
$$
\mathscr{P}'(0)=\{S^{j}(C(i,0)):1\le i\le t(n_{0}),\, 0\le j<h_{i}(n_{0})\}
$$
where $S^{j}(C(i,0))=C(i,j)$. Refine $\mathscr{P}'(0)$ with respect to each clopen set in $\mathcal{P}$ as in Proposition~\ref{K-R refine}, thus preserving the tower structure, and call the result $\mathscr{P}(0)$. So $X_{1}$ now looks like
\vskip 2mm
$$
\begin{tikzpicture}
\draw(-1,0)--(-2/3,0) node[midway,below]{$E_{1}(0)$};
\draw[->](-5/6,0)--(-5/6,1) node[midway,right]{$S$};
\node at (-5/6,1.5) []{$\vdots$};
\draw(-1,2)--(-2/3,2);
\draw(-1,1)--(-2/3,1);
\node at (-.15,0) [] {$\dots$};
\draw(1/3,0)--(2/3,0) node[midway,below]{$E_{m}(0)$};
\draw[->](.5,0)--(.5,1) node[midway,right]{$S$};
\draw((1/3,1)--(2/3,1);
\node at (.5,1.5) []{$\vdots$};
\draw(1/3,2)--(2/3,2);
\draw(5/3,0)--(2,0) node[midway,below]{$E_{m+1}(0)$};
\draw[->](11/6,0)--(11/6,1) node[midway,right]{$S$};
\draw(5/3,1)--(2,1);
\node at (11/6,1.5) []{$\vdots$};
\draw(5/3,2)--(2,2);
\draw[->](11/6,2)--(11/6,3) node[midway,right]{$S$};
\draw(5/3,3)--(2,3);
\node at (2.75,0) [] {$\dots$};
\draw(10/3,0)--(11/3,0) node[midway,below]{$E_{n}(0)$};
\draw[->](21/6,0)--(21/6,1) node[midway,right]{$S$};
\draw(10/3,1)--(11/3,1);
\node at (21/6,1.5)[]{$\vdots$};
\draw(10/3,2)--(11/3,2);
\draw[->](21/6,2)--(21/6,3) node[midway,right]{$S$};
\draw(10/3,3)--(11/3,3);
\node at (4.25,0) [] {$\dots$};
\draw(14/3,0)--(5,0) node[midway,below]{$E_{p}(0)$};
\draw[->](29/6,0)--(29/6,1) node[midway,right]{$S$};
\draw(14/3,1)--(5,1);
\node at (29/6,1.5)[] {$\vdots$};
\draw(14/3,2)--(5,2);
\draw[->](29/6,2)--(29/6,3) node[midway,right]{$S$};
\draw(14/3,3)--(5,3);
\draw[->](29/6,3)--(29/6,4) node[midway,right]{$S$};
\draw(14/3,4)--(5,4);
\node at (5.5,0) []{$\dots$};
\draw(6,0)--(19/3,0) node[near end,below]{$E_{t(n'_{0})}(0)$};
\draw[->](37/6,0)--(37/6,1) node[midway,right]{$S$};
\draw(6,1)--(19/3,1);
\node at (37/6,1.5)[]{$\vdots$};
\draw(6,2)--(19/3,2);
\draw[->](37/6,2)--(37/6,3) node[midway,right]{$S$};
\draw(6,3)--(19/3,3);
\draw[->](37/6,3)--(37/6,4) node[midway,right]{$S$};
\draw(6,4)--(19/3,4);
\end{tikzpicture}
$$
\vskip 2mm
where for each $i$
$$
C(i,0)=\disUnion_{j=k_{i}}^{k_{i+1}-1}E_{j}(0)\text{ and }\mathscr{P}(0)=\{S^{j}E_{i}(0):1\le i\le t'(n_{0}), 0\le j<h'_{i}(n_{0})\}
$$
where $t'(n_{0})$ is the new number of base levels and $h_{i}'(n_{0})$ gives the height of the respective column. Because $\mu(C(i,0))=\mu(C'(i,0))$ for all $\mu\in M(X_{1},T_{1})$, $F:X_{1}\rightarrow X_{2}$ is a homeomorphism and through the use of Lemma~\ref{cpartition} we can refine $\mathcal{Q}(n_{0})$, our tower partition in $X_{2}$ to look exactly like $\mathscr{P}(0)$. That is, there are sets $B'_{j}(0)$ such that
$$
\mu\circ F^{-1}(B'_{j}(0))=\mu(E_{j}(0))\text{ and } B_{i}(n_{0})=\disUnion_{\ell=k_{i}}^{k_{i+1}-1}B'_{\ell}(0)
$$
and set 
$$
\mathcal{Q}'(n_{0})=\{T_{2}^{j}B'_{\ell}(0):1\le\ell\le t'(n_{0}), 0\le j<h'_{i}(n_{0})\}.
$$
Hence, $X_{2}$ looks like 
\vskip 2mm
$$
\begin{tikzpicture}
\draw(-1,0)--(-2/3,0) node[midway,below]{$B'_{1}(0)$};
\draw[->](-5/6,0)--(-5/6,1) node[midway,right]{$T_{2}$};
\node at (-5/6,1.5) []{$\vdots$};
\draw(-1,2)--(-2/3,2);
\draw(-1,1)--(-2/3,1);
\node at (-.15,0) [] {$\dots$};
\draw(1/3,0)--(2/3,0) node[midway,below]{$B'_{m}(0)$};
\draw[->](.5,0)--(.5,1) node[midway,right]{$T_{2}$};
\draw((1/3,1)--(2/3,1);
\node at (.5,1.5) []{$\vdots$};
\draw(1/3,2)--(2/3,2);
\draw(11/6,0)--(13/6,0) node[midway,below]{$B'_{m+1}(0)$};
\draw[->](2,0)--(2,1) node[midway,right]{$T_{2}$};
\draw(11/6,1)--(13/6,1);
\node at (2,1.5) []{$\vdots$};
\draw(11/6,2)--(13/6,2);
\draw[->](2,2)--(2,3) node[midway,right]{$T_{2}$};
\draw(11/6,3)--(13/6,3);
\node at (2.75,0) [] {$\dots$};
\draw(10/3,0)--(11/3,0) node[midway,below]{$B'_{n}(0)$};
\draw[->](21/6,0)--(21/6,1) node[midway,right]{$T_{2}$};
\draw(10/3,1)--(11/3,1);
\node at (21/6,1.5)[]{$\vdots$};
\draw(10/3,2)--(11/3,2);
\draw[->](21/6,2)--(21/6,3) node[midway,right]{$T_{2}$};
\draw(10/3,3)--(11/3,3);
\node at (4.25,0) [] {$\dots$};
\draw(14/3,0)--(5,0) node[midway,below]{$B'_{q}(0)$};
\draw[->](29/6,0)--(29/6,1) node[midway,right]{$T_{2}$};
\draw(14/3,1)--(5,1);
\node at (29/6,1.5)[] {$\vdots$};
\draw(14/3,2)--(5,2);
\draw[->](29/6,2)--(29/6,3) node[midway,right]{$T_{2}$};
\draw(14/3,3)--(5,3);
\draw[->](29/6,3)--(29/6,4) node[midway,right]{$T_{2}$};
\draw(14/3,4)--(5,4);
\node at (5.5,0) []{$\dots$};
\draw(6,0)--(19/3,0) node[near end,below]{$B'_{t'(n_{0})}(0)$};
\draw[->](37/6,0)--(37/6,1) node[midway,right]{$T_{2}$};
\draw(6,1)--(19/3,1);
\node at (37/6,1.5)[]{$\vdots$};
\draw(6,2)--(19/3,2);
\draw[->](37/6,2)--(37/6,3) node[midway,right]{$T_{2}$};
\draw(6,3)--(19/3,3);
\draw[->](37/6,3)--(37/6,4) node[midway,right]{$T_{2}$};
\draw(6,4)--(19/3,4);
\end{tikzpicture}
$$
\vskip 2mm
and we define a map on the level of sets, which in the limit will give us our conjugacy. Define $\Phi_{0}:\mathscr{P}(0)\rightarrow\mathcal{Q}'(n_{0})$ by
$$
\Phi_{0}(S^{j}(E_{i}(0)))=T_{2}^{j}(B'_{i}(0)).
$$
We have now completed the first step of our construction!

\textbf{Inductive step} 

We now move onto the second (inductive) step of our construction. Let $\varepsilon_{1}=\min\{\alpha_{1},\rho_{0}\}$ where
\begin{align*}
\alpha_{1}&=\min\left\{\int_{X}\mathbbm{1}_{A_{1}}\,d\mu:\mu\in M(X_{1},T_{1})\right\}>0\\
\rho_{0}&=\displaystyle\min_{1\le i\le t'(n_{0})}\left\{\int_{X}\mathbbm{1}_{E_{i}(0)}\,d\mu:\mu\in M(X_{1},T_{1})\right\}>0
\end{align*}
and find $n_{1}>n_{0}$ large enough such that the following are true:
\begin{enumerate}
\item for every $\nu\in M(X_{2},T_{2})$
$$
\nu\left(\Union_{i=1}^{t(n_{1})}B_{i}(n_{1})\right)<\varepsilon_{1}
$$
\item $\mathcal{Q}(n_{1})$ refines $\mathcal{Q}'(n_{0})$ i.e. $\mathcal{Q}(n_{1})\ge \mathcal{Q}'(n_{0})$.
\end{enumerate}
Now as $\mathcal{Q}(n_{1})\ge\mathcal{Q}'(n_{0})$ we see that each column in $\mathcal{Q}(n_{1})$ is simply made up of stacking towers from $\mathcal{Q}'(n_{0})$ upon one another. So we view $\mathcal{Q}(n_{1})$ not only as a space time partition, but also as a labeled or tagged partition by the previous tower construction, in this case tagged by the towers of $\mathcal{Q}'(n_{0})$. We give a picture as an illustrative example of the tagging or labeling of the towers.
$$
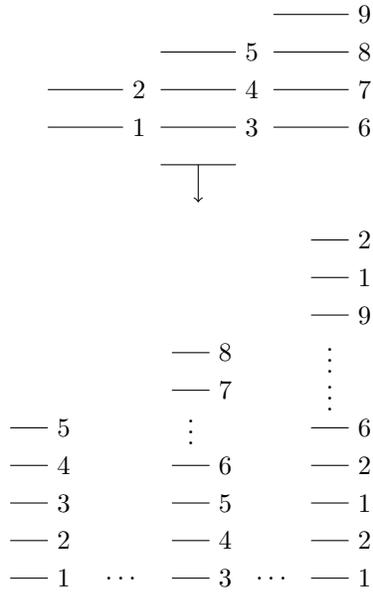
\begin{tikzpicture}
\draw(-1,0)--(0,0) node[right]{$1$};
\draw(-1,.5)--(0,.5) node[right]{$2$};
\draw(.5,0)--(1.5,0) node[right]{$3$};
\draw(.5,.5)--(1.5,.5) node[right]{$4$};
\draw(.5,1)--(1.5,1) node[right]{$5$};
\draw(2,0)--(3,0) node[right]{$6$};
\draw(2,.5)--(3,.5) node[right]{$7$};
\draw(2,1)--(3,1) node[right]{$8$};
\draw(2,1.5)--(3,1.5) node[right]{$9$};
\draw(.5,-.5)--(1.5,-.5);
\draw[->](1,-.5)--(1,-1);
\draw(2.5,-1.5)--(3,-1.5) node[right]{$2$};
\draw(2.5,-2)--(3,-2) node[right]{$1$};
\draw(2.5,-2.5)--(3,-2.5) node[right]{$9$};
\node at (2.75,-3)[]{$\vdots$};
\node at (2.75,-3.5)[]{$\vdots$};
\draw(2.5,-4)--(3,-4) node[right]{$6$};
\draw(2.5,-4.5)--(3,-4.5) node[right]{$2$};
\draw(2.5,-5)--(3,-5) node[right]{$1$};
\draw(2.5,-5.5)--(3,-5.5) node[right]{$2$};
\draw(2.5,-6)--(3,-6) node[right]{$1$};
\node at (2,-6)[]{$\dots$};
\draw(.65,-6)--(1.15,-6) node[right]{$3$};
\draw(.65,-5.5)--(1.15,-5.5) node[right]{$4$};
\draw(.65,-5)--(1.15,-5) node[right]{$5$};
\draw(.65,-4.5)--(1.15,-4.5) node[right]{$6$};
\node at (.9,-3.95) []{$\vdots$};
\draw(.65,-3.5)--(1.15,-3.5) node[right]{$7$};
\draw(.65,-3)--(1.15,-3) node[right]{$8$};
\node at (0,-6)[]{$\dots$};
\draw(-1.5,-6)--(-1.0,-6) node[right]{$1$};
\draw(-1.5,-5.5)--(-1,-5.5) node[right]{$2$};
\draw(-1.5,-5)--(-1,-5) node[right]{$3$};
\draw(-1.5,-4.5)--(-1,-4.5) node[right]{$4$};
\draw(-1.5,-4)--(-1,-4) node[right]{$5$};
\end{tikzpicture}
$$
As a consequence of $\mathcal{Q}(n_{1})\ge \mathcal{Q}'(n_{0})$, we see that for $1\le i\le t'(n_{0})$ and $0\le j<h'_{i}(n_{0})$ we have
$$
T_{2}^{j}(B'_{i}(0))=\disUnion_{a=1}^{m}T_{2}^{j_{a}}(B_{i_{a}}(n_{1}))
$$
and thus we have that
$$
F^{-1}(T_{2}^{j}(B'_{i}(0)))=\disUnion_{a=1}^{m} F^{-1}(T_{2}^{j_{a}}(B_{i_{a}}(n_{1}))).
$$
As a result of Lemma~\ref{cpartition} we can write
$$
S^{j}(E_{i}(0))=\disUnion_{a=1}^{m}E_{(i,a)}(1).
$$
Using Lemma~\ref{cpartition} on each copied tower of $X_{2}$ in $X_{1}$, we can copy $\mathcal{Q}(n_{1})$ in $X_{1}$ in a way which refines our $\mathscr{P}(0)$: call this collection $\mathscr{P}'(1)$. Recall, we have already defined $S$ on a large portion of $X_{1}$ and we do not need, nor want, to be redefining $S$ on this portion of the space. Following the tagging from $\mathcal{Q}(n_{1})$, extend $S$ on any and all previous undefined pieces, save for the top levels of each column. As before, using Proposition~\ref{K-R refine} refine $\mathscr{P}'(1)$ with respect to each clopen set in $\mathcal{P}(1)$ and call $\mathscr{P}(1)$ the result of this refinement. Specfically,
$$
\mathscr{P}(1)=\{S^{j}(E_{i}(1)):1\le i\le t'(n_{1}), 0\le j<h'_{i}(n_{1})\}.
$$
Use Lemma~\ref{cpartition} and $F$ to push this refinement onto $\mathcal{Q}(n_{1})$, resulting in
$$
\mathcal{Q}'(n_{1})=\{T_{2}^{j}(B'_{i}(1)):1\le i\le t'(n_{1}), 0\le j<h'_{i}(n_{1})\}.
$$
As before we define $\Phi_{1}:\mathscr{P}(1)\rightarrow\mathcal{Q}'(n_{1})$ by
$$
\Phi_{1}(S^{j}(E_{i}(1)))=T_{2}^{j}(B'_{i}(1)),
$$
and by construction $\Phi_{1}$ extends $\Phi_{0}$. We continue this process and thus by induction we see that we will have defined $S:X_{1}\backslash\{T_{1}^{-1}x\}\rightarrow X_{1}\backslash\{x\}$. By construction, $S$ is a homeomorphism and so by defining $p(T_{1}^{-1}x)=1$ we see that $S$ now lifts to a homeomorphism on all of $X$. Furthermore, $\{\Phi_{n}\}_{n\ge 0}$ induces, by way of intersection, a point map $\varphi:X_{1}\rightarrow X_{2}$, which is our conjugacy from $(X_{1},S)$ onto $(X_{2},T_{2})$. The fact $\varphi$ is well defined and a homeomorphism is due to both $\{\mathscr{P}(k)\}_{k\ge 0}$ and $\{\mathcal{Q}'(n_{k})\}_{k\ge 0}$ being generating for the topology of $X_{1}$ and $X_{2}$ respectively. Moreover, $\varphi$ conjugates $S$ and $T_{2}$ is built into the definition of each $\Phi_{n}$ and each $\Phi_{n+1}$ extends the previous $\Phi_{n}$. Therefore, our theorem as been proved.
\end{proof}
\section{Speedup Equivalence}
We wish to view speedups as a relation and to that end it will be helpful to introduce some notation. Let $(X_{i},T_{i}),\, i=1,2$ be minimal Cantor systems and write $T_{1}\rightsquigarrow T_{2}$ to mean that $(X_{2},T_{2})$ is a speedup of $(X_{1},T_{1})$. Moreover, define $(X_{1},T_{1})$ and $(X_{2},T_{2})$ to be \emph{\textbf{speedup equivalent}}, written $T_{1}\leftrightsquigarrow T_{2}$, if and only if $T_{1}\rightsquigarrow T_{2}$ and $T_{2}\rightsquigarrow T_{1}$. It is straight forward to verify that speedup equivalence is indeed an equivalence relation. Combining \cite[Thm 2.2]{GPS} with our main theorem we obtain the following corollary.

\begin{corollary}\label{Cor.1}
	Let $(X_{1},T_{1})$ and $(X_{2},T_{2})$ be minimal Cantor systems. If $(X_{1},T_{1})$ and $(X_{2},T_{2})$ are orbit equivalent, then $(X_{1},T_{1})$ and $(X_{2},T_{2})$ are speedup equivalent (i.e. $T_{1}\leftrightsquigarrow T_{2})$.
\end{corollary}

Rephrasing Corollary~\ref{Cor.1} above, as equivalence relations orbit equivalence is contained in speedup equivalence. This leads us to a fundamental question: are orbit equivalence and speedup equivalence the same equivalence relation? At his time we only have the partial answer in the form of the Theorem~\ref{speedupequivalence}. However, before we can prove the aforementioned theorem we need a proposition, for which the proof is straight forward and hence omitted.
\begin{proposition}\label{MutuallySingular}
	Let $(X_{i},T_{i})$ be minimal Cantor systems and $\varphi:X_{1}\rightarrow X_{2}$ be a homeomorphism. If $\varphi_{*}:M(X_{1},T_{1})\hookrightarrow M(X_{2},T_{2})$ is an injection, then $\varphi_{*}$ preserves pairs of mutually singular measures.
\end{proposition}	
We now use Proposition~\ref{MutuallySingular} to prove the following theorem.
\begin{theorem}\label{speedupequivalence}
	Let $(X_{i},T_{i}),\, i=1,2,$ be minimal Cantor systems each with finitely many ergodic measures. If $(X_{1},T_{1})$ and $(X_{2},T_{2})$ are speedup equivalent, then $(X_{1},T_{1})$ and $(X_{2},T_{2})$ are orbit equivalent.
\end{theorem}
\begin{proof}
	Since $T_{1}\leftrightsquigarrow T_{2}$, combining part $(3)$ of Theorem~\ref{Main Theorem} and Proposition~\ref{MutuallySingular} it follows immediately that 
    $$
    |\partial_{e}(M(X_{1},T_{1}))|=|\partial_{e}(M(X_{2},T_{2}))|
    $$ 
    and without loss of generality we may assume $|\partial_{e}(M(X_{1},T_{1}))|=n$ for some $n\in\Z^{+}$. Every measure $\mu$ in $M(X_{2},T_{2})$ is a convex combination of ergodic measures in a unique way. With this in mind for $\mu\in M(X_{2},T_{2})$, let $E(\mu)$ denote the collection of all ergodic measures of $M(X_{2},T_{2})$ which have a positive coefficient in the unique ergodic decomposition of $\mu$. Observe if $\mu_{1},\mu_{2}\in M(X_{2},T_{2})$ with $\mu_{1}\neq\mu_{2}$ and $\mu_{1}\perp\mu_{2}$, then
	$$
	E(\mu_{1})\intersect E(\mu_{2})=\emptyset.
	$$
	Now as $T_{1}\rightsquigarrow T_{2}$ there exists $\varphi:X_{1}\rightarrow X_{2}$, a homeomorphism, such that 
	$$
	\varphi_{*}:M(X_{1},T_{1})\hookrightarrow M(X_{2},T_{2})
	$$
	is an injection. Since $|\partial_{e}(M(X_{1},T_{1}))|=|\partial_{e}(M(X_{2},T_{2}))|=n$ and $\varphi_{*}$ is injective, we see that $\{E(\varphi_{*}(\nu_{i}))\}_{i=1}^{n}$, where $\{\nu_{i}\}_{i=1}^{n}=\partial_{e}(M(X_{1},T_{1}))$, is a collection of $n$ pairwise disjoint sets, as distinct ergodic measures are mutually singular. It follows that for each $i=1,2,\dots,n$, $E(\nu_{i})$ is a distinct singleton, and therefore $\varphi_{*}(\partial_{e}(M(X_{1},T_{1})))=\partial_{e}(M(X_{2},T_{2}))$. Coupling the facts that $\varphi_{*}$ is an affine map and a bijection on extreme points, we may conclude that $\varphi_{*}$ is a bijection, and hence is an affine homeomorphism between $M(X_{1},T_{1})$ and $M(X_{2},T_{2})$ arising from a space homeomorphism. Therefore, by \cite[Thm. $2.2$]{GPS} $(X_{1},T_{1})$ and $(X_{2},T_{2})$ are orbit equivalent.
\end{proof}

There are two obstacles which arise when trying to extend Theorem~\ref{speedupequivalence} to the infinite dimensional case. The first is whether or not it is always true that
$$
\varphi_{*}(\partial_{e}(M(X_{1},T_{2})))\subseteq\partial_{e}(M(X_{2},T_{2}))
$$
whenever $T_{1}\rightsquigarrow T_{2}$. The second is whether the Schr$\ddot{o}$der-Bernstein Theorem holds in the category of simple dimension groups with our morphisms. The Schr$\ddot{o}$der-Bernstein Theorem for dimension groups and simple dimension groups was addressed in the Glasner and Weiss paper \cite{Glasner-Weiss}, which we discuss below.

We must remark that speedup equivalence looks quite similar to weak orbit equivalence, especially in terms of weakly isomorphic dimension groups. Observe that we have surjective homomorphisms and the key difference is that we require our homomorphism to exhaust the positive cone in the image space. We mention this here because one avenue to try to answer the speedup equivalence question would be to show that given two dimension groups $(G_{1},G_{1}^{+},\textbf{1})$ and $(G_{2},G_{2}^{+},\textbf{1})$ with surjective group homomorphisms $\varphi_{1},\varphi_{2}$ satisfying 
\begin{align*}
\varphi_{1}&:G_{1}\rightarrow G_{2}\text{ and $\varphi_{1}(G_{1}^{+})=G_{2}^{+},\, \varphi_{1}(\textbf{1})=\textbf{1}$} \\
\varphi_{2}&:G_{2}\rightarrow G_{1}\text{ and $\varphi_{2}(G_{2}^{+})=G_{1}^{+},\, \varphi_{2}(\textbf{1})=\textbf{1}$},
\end{align*}
then in fact $(G_{1},G_{1}^{+},\textbf{1})\cong(G_{2},G_{2}^{+},\textbf{1})$. However, Glasner and Weiss, in \cite{Glasner-Weiss}, gave a beautiful counter example, Example $4.2$, which shows that even if $Inf\,G=0$ the Schr$\ddot{o}$eder-Bernstein theorem fails for simple dimension groups. Unfortunately, their example fails to exhaust the positive cone. Since this cannot happen with speedups, this example would need some modification to apply. 
\section{Example}
We will now use the main theorem, Theorem~\ref{Main Theorem}, to show the aforementioned claim: speedups can leave the conjugacy class and even the orbit equivalence class of the original system. This will be demonstrated by showing that the simplex of invariant Borel probability measure can grow: see Proposition~\ref{Simplex inequality}. To do so we will use Theorem~\ref{Main Theorem} in conjunction with Theorem~\ref{measures-states}, which recall says that states and invariant measures are in bijective correspondence. 

Let $(X,T)$ be the dyadic odometer. It is well known that the dimension group associated to this system is $(\Z[\frac{1}{2}],\Z[\frac{1}{2}]^{+},\textbf{1})$.  Since $(X,T)$ is uniquely ergodic, by Proposition~\ref{measures-states} it follows that $(\Z[\frac{1}{2}],\Z[\frac{1}{2}]^{+},\textbf{1})$ has only one state. Our goal is to construct a simple dimension group with two states, such that it factors onto $(\Z[\frac{1}{2}],\Z[\frac{1}{2}]^{+},\textbf{1})$ in the sense of the main theorem. One can show that the following is a dimension group
\begin{center}
$(\Z[\frac{1}{2}]\oplus\Z[\frac{1}{2}],\Z[\frac{1}{2}]^{++}\oplus\Z[\frac{1}{2}]^{++}\union\{(0,0)\},(1,1))$
\end{center}
where 
\begin{center}
$\Z[\frac{1}{2}]^{++}=\{x\in\Z[\frac{1}{2}]:x>0\}$.
\end{center}
Note, the Riesz interpolation property is satisfied as $\Z[\frac{1}{2}]$ is a totally ordered set. To see that $(\Z[\frac{1}{2}]\oplus\Z[\frac{1}{2}],\Z[\frac{1}{2}]^{++}\oplus\Z[\frac{1}{2}]^{++}\union\{(0,0)\}, (1,1))$ is a simple dimension group we use the following lemma from \cite{Goodearl}.
\begin{lemma}\cite[Lemma $14.1$]{Goodearl}\label{simple}
Let $G$ be a nonzero directed Abelian group. Then $G$ is simple if and only if every nonzero element of $G^{+}$ is an order-unit in $G$.
\end{lemma}
An ordered group $G$ is \emph{directed} if all elements of $G$ have the form $x-y$ for some $x,y\in G^{+}$ and so any dimension group is directed. One can use Lemma~\ref{simple} to show that $(\Z[\frac{1}{2}]\oplus\Z[\frac{1}{2}],\Z[\frac{1}{2}]^{++}\oplus\Z[\frac{1}{2}]^{++}\union\{(0,0)\}, (1,1))$ is indeed a simple dimension group. Now by Theorem~\ref{DimGpReal} there exists minimal Cantor system $(X_{2},T_{2})$ such that 
\begin{center}
$(K^{0}(X_{2},T_{2}),K^{0}(X_{2},T_{2})^{+},\textbf{1})\cong(\Z[\frac{1}{2}]\oplus\Z[\frac{1}{2}],\Z[\frac{1}{2}]^{++}\oplus\Z[\frac{1}{2}]^{++}\union\{(0,0)\}, (1,1))$.
\end{center}
In addition, one can verify that
\begin{center}
$\pi_{1}:(\Z[\frac{1}{2}]\oplus\Z[\frac{1}{2}],\Z[\frac{1}{2}]^{++}\oplus\Z[\frac{1}{2}]^{++}\union\{(0,0)\}, (1,1))\rightarrow(\Z[\frac{1}{2}],\Z[\frac{1}{2}]^{+},1)$
\end{center}
satisfies condition $(2)$ of the main theorem, Theorem~\ref{Main Theorem}, whence $(X_{2},T_{2})$ is a speedup of $(X,T)$ the dyadic odometer. 

We will now show $(X_{2},T_{2})$ and $(X,T)$ are not conjugate to one another, hence speedups can leave their conjugacy classes. Furthermore, we will actually show that $(X_{2},T_{2})$ and $(X,T)$ cannot even be orbit equivalent. To see this, it suffices to show, by \cite[Theorem $2.2$]{GPS}, that their respective dimension groups modulo infinitesimals are not isomorphic as dimension groups. To accomplish this we will use states and show $(\Z[\frac{1}{2}]\oplus\Z[\frac{1}{2}],\Z[\frac{1}{2}]^{++}\oplus\Z[\frac{1}{2}]^{++}\union\{(0,0)\}, (1,1))$ and $(\Z[\frac{1}{2}],\Z[\frac{1}{2}]^{+},1)$ have different state spaces. Recall all states can be realized as integration against invariant probability measures, hence as $(X,T)$ is uniquely ergodic it has exactly one state, namely the identity map. Thus, it suffices to show that $(X_{2},T_{2})$ has more than one invariant measure, or more to the point, that $(\Z[\frac{1}{2}]\oplus\Z[\frac{1}{2}],\Z[\frac{1}{2}]^{++}\oplus\Z[\frac{1}{2}]^{++}\union\{(0,0)\}, (1,1))$ has more than one state.

Before we begin we must deal with one technical aspect, that is, we know that $(K^{0}(X_{2},T_{2}),K^{0}(X_{2},T_{2})^{+},\textbf{1})$ is isomorphic as a dimension group to $(\Z[\frac{1}{2}]\oplus\Z[\frac{1}{2}],\Z[\frac{1}{2}]^{++}\oplus\Z[\frac{1}{2}]^{++}\union\{(0,0)\}, (1,1))$, so we must show that this group has only trivial infinitesimals. Recall infinitesimals evaluate to $0$ for every state on the dimension group, and 
\begin{center}
$\pi_{i}:(\Z[\frac{1}{2}]\oplus\Z[\frac{1}{2}],\Z[\frac{1}{2}]^{++}\oplus\Z[\frac{1}{2}]^{++}\union\{(0,0)\}, (1,1))\rightarrow\R$
\end{center}
$i=1,2$ are states. From this we can deduce that the only infinitesimal of $(\Z[\frac{1}{2}]\oplus\Z[\frac{1}{2}],\Z[\frac{1}{2}]^{++}\oplus\Z[\frac{1}{2}]^{++}\union\{(0,0)\}, (1,1))$ is $(0,0)$, hence
\begin{center}
$(K^{0}(X_{2},T_{2})/Inf(K^{0}(X_{2},T_{2})),K^{0}(X_{2},T_{2})^{+}/Inf(K^{0}(X_{2},T_{2})),\textbf{1})\cong(\Z[\frac{1}{2}]\oplus\Z[\frac{1}{2}],\Z[\frac{1}{2}]^{++}\oplus\Z[\frac{1}{2}]^{++}\union\{(0,0)\}, (1,1))$
\end{center}
as dimension groups. Furthermore, as $\pi_{1}\neq\pi_{2}$, on $(\Z[\frac{1}{2}]\oplus\Z[\frac{1}{2}],\Z[\frac{1}{2}]^{++}\oplus\Z[\frac{1}{2}]^{++}\union\{(0,0)\}, (1,1))$, we have
\begin{center}
$(\Z[\frac{1}{2}],\Z[\frac{1}{2}]^{+},1)\ncong(\Z[\frac{1}{2}]\oplus\Z[\frac{1}{2}],\Z[\frac{1}{2}]^{++}\oplus\Z[\frac{1}{2}]^{++}\union\{(0,0)\}, (1,1))$
\end{center}
and so $(X,T)$ and $(X_{2},T_{2})$ are not orbit equivalent, hence not conjugate.

\end{document}